\documentclass[a4paper,11pt,reqno]{amsart}
\usepackage{amssymb}

\newtheorem{theorem}{Theorem}
\newtheorem{lemma}[theorem]{Lemma}
\newtheorem{corollary}[theorem]{Corollary}
\newtheorem{proposition}[theorem]{Proposition}
\numberwithin{theorem}{section}
\numberwithin{equation}{section}

\theoremstyle{definition}
\newtheorem{definition}[theorem]{Definition}

\newtheorem{remark}[theorem]{Remark}
\newtheorem{examples}[theorem]{Example}
\newtheorem*{remark*}{Remark}

\newcommand{\Z}{{\mathbb Z}}

\newcommand{\C}{{\mathbb C}}
\newcommand{\Q}{{\mathbb Q}}

\newcommand{\cA}{{\mathcal A}}

\newcommand{\cC}{{\mathcal C}}

\newcommand{\J}{{\mathcal J}}
\newcommand{\cJ}{{\mathcal J}}
\newcommand{\cK}{{\mathcal K}}

\newcommand{\cT}{{\mathcal T}}
\newcommand{\cQ}{{\mathcal Q}}
\newcommand{\cO}{{\mathcal O}}
\newcommand{\cV}{{\mathcal V}}

\newcommand{\g}{{\mathfrak g}}
\newcommand{\gl}{{\mathfrak{gl}}}

\newcommand{\h}{{\mathfrak h}}
\newcommand{\m}{{\mathfrak m}}
\newcommand{\s}{{\mathfrak s}}
\newcommand{\rs}{{\mathfrak r}}

\newcommand{\frso}{{\mathfrak{so}}}  
\newcommand{\so}{{\mathfrak{so}}}
\newcommand{\frsp}{{\mathfrak{sp}}}

\newcommand{\frsl}{{\mathfrak{sl}}}  
\newcommand{\spl}{{\mathfrak{sl}}}

\newcommand{\anti}{{\mathfrak{skew}}}
\newcommand{\sym}{{\mathfrak{sym}}}

 \DeclareMathOperator{\eespan}{span}
 \DeclareMathOperator{\End}{End}
 \DeclareMathOperator{\Hom}{Hom}

 \DeclareMathOperator{\tr}{tr}
 \DeclareMathOperator{\Ad}{Ad}
 \DeclareMathOperator{\ad}{ad}
 \DeclareMathOperator{\Der}{Der}
  \DeclareMathOperator{\Inder}{Inder}
 \DeclareMathOperator{\Mat}{Mat}

 \def\dimens{\mathop{\rm dim}\nolimits}

\def\bigstrut{\vrule height 12pt width 0ptdepth 2pt}

\def\hreglon{\hrule height1pt}
\def\vreglon{\vrule height 12pt width1pt depth 4pt}

\def\hreglonfill{\leaders\hreglon\hfill}

\begin{document}

\title{Irreducible Lie-Yamaguti algebras}

\author{Pilar Benito}

\thanks{Supported by the
Spanish Ministerio de Ciencia y Tecnolog\'{\i}a and FEDER (BFM
2001-3239-C03-02,03) and the Ministerio de Educaci\'on y Ciencia and
FEDER (MTM 2004-08115-C04-02 and MTM 2007-67884-C04-02,03). Pilar Benito and Fabi\'an
Mart\'{\i}n-Herce also acknowledge support from the  Comunidad
Aut\'onoma de La Rioja (ANGI2005/05,06),
 and Alberto Elduque from
the Diputaci\'on General de Arag\'on (Grupo de Investigaci\'on de
\'Algebra).}

\address{Departamento de Matem\'aticas y Computaci\'on, Universidad de
La Rioja, 26004 Logro\~no, Spain}

\email{pilar.benito@unirioja.es}



\author{Alberto Elduque}

\address{Departamento de Matem\'aticas e Instituto Universitario de Matem\'aticas y Aplicaciones, Universidad de
Zaragoza, 50009 Zaragoza, Spain}

\email{elduque@unizar.es}

\author{Fabi\'an Mart\'in-Herce}

\address{Departamento de Matem\'aticas y Computaci\'on, Universidad de
La Rioja, 26004 Logro\~no, Spain}

\email{fabian.martin@dmc.unirioja.es}

\date{\today}

\subjclass[2000]{Primary 17A30, 17B60}

\keywords{Lie-Yamaguti algebra, irreducible, Tits construction}

\begin{abstract}
Lie-Yamaguti algebras (or generalized Lie triple systems) are
binary-ternary algebras intimately related to reductive homogeneous
spaces. The Lie-Yamaguti algebras which are irreducible as modules
over their Lie inner derivation algebra are the algebraic
counterpart of the isotropy irreducible homogeneous spaces.

These systems will be shown to split into three disjoint types:
adjoint type, non-simple type and generic type. The systems of the
first two types will be classified and most of them will be shown to
be related to a Generalized Tits Construction of Lie algebras.
\end{abstract}

\maketitle


\section{Introduction}

Let $G$ be a connected Lie group with Lie algebra $\g$, $H$ a closed
subgroup of $G$, and let $\h$ be the associated subalgebra of $\g$.
The corresponding homogeneous space $M=G/H$ is said to be
\emph{reductive} (\cite[\S 7]{Nom54}) in case there is a subspace
$\m$ of $\g$ such that $\g=\h\oplus\m$ and $\Ad(H)(\m)\subseteq \m$.

In this situation, Nomizu proved \cite[Theorem 8.1]{Nom54} that
there is a one-to-one correspondence between the set of all
$G$-invariant affine connections on $M$ and the set of bilinear
multiplications $\alpha:\m\times\m\rightarrow\m$ such that the
restriction of $\Ad(H)$ to $\m$ is a subgroup of the automorphism
group of the nonassociative algebra $(\m,\alpha)$.

There exist natural binary and ternary products defined in $\m$,
given by
\begin{equation}\label{eq:binter}
\begin{split}
&x\cdot y = \pi_{\m}\bigl([x,y]\bigr),\\
&[x,y,z]=\bigl[\pi_{\h}([x,y]),z],
\end{split}
\end{equation}
for any $x,y,z\in\m$, where $\pi_{\h}$ and $\pi_{\m}$ denote the
projections on $\h$ and $\m$ respectively, relative to the reductive
decomposition $\g=\h\oplus\m$. Note that the condition
$\Ad(H)(\m)\subseteq \m$ implies the condition $[\h,\m]\subseteq
\m$, the converse being valid if $H$ is connected.

There are two distinguished invariant affine connections: the
 natural connection (or canonical connection of the first kind),
which corresponds to the bilinear multiplication given by
$\alpha(x,y)=\frac{1}{2} x\cdot y$ for any $x,y\in\m$, which has
trivial torsion, and the canonical connection corresponding to the
trivial multiplication: $\alpha(x,y)=0$ for any $x,y\in\m$. In case
the reductive homogeneous space is symmetric, so $[\m,\m]\subseteq
\h$, these two connections coincide. For the canonical connection,
the torsion and curvature tensors are given on the tangent space to
the point $eH\in M$ ($e$ denotes the identity element of $G$), which
can be naturally identified with $\m$, by
\[
T(x,y)=-x\cdot y,\qquad R(x,y)z=-[x,y,z],
\]
for any $x,y,z\in \m$ (see \cite[Theorem 10.3]{Nom54}).

Moreover, Nomizu showed too that the affine connections on manifolds
wiht parallel torsion and curvature are locally equivalent to
canonical connections  on reductive homogeneous spaces.

Yamaguti \cite{Yam58} considered the properties of the torsion and
curvature of these canonical connections (or alternatively, of the
binary and ternary multiplications in \eqref{eq:binter}), and thus
defined what he called the \emph{general Lie triple systems}, later
renamed as \emph{Lie triple algebras} in \cite{Kik75}. We will
follow here the notation in \cite[Definition 5.1]{KinWei}, and will
call these systems \emph{Lie-Yamaguti algebras}:

\begin{definition}\label{df:LY}
A \emph{Lie-Yamaguti algebra} $(\m,x\cdot y,[x\,,y\,,z\,])$  ({\em
LY-algebra} for short) is a vector space $\m$ equipped with a
bilinear operation $\cdot : \m\times\m\rightarrow \m$ and a
trilinear operation $[\,,\,,\,]: \m\times\m\times\m\rightarrow \m$
such that, for all $x,y,z,u,v,w\in \m$:
\begin{enumerate}
\item[(LY1)] $x\cdot x=0$,
\item[(LY2)] $[x,x,y]=0$,
\item[(LY3)] $\sum_{(x,y,z)}\Bigl([x,y,z]+(x\cdot y)\cdot
z\Bigr)=0$,
\item[(LY4)] $\sum_{(x,y,z)}[x\cdot y,z,t]=0$,
\item[(LY5)] $[x, y,u\cdot v]=[x,y,u]\cdot v+u\cdot [x,y,v]$,
\item[(LY6)]
$[x,y,[u,v,w]]=[[x,y,u],v,w]+[u,[x,y,v],w]+[u,v,[x,y,w]]$.
\end{enumerate}
\end{definition}

\noindent Here $\sum_{(x,y,z)}$ means the cyclic sum on $x,y,z$.
\smallskip

The LY-algebras with $x\cdot y=0$ for any $x,y$ are exactly the Lie
triple systems, closely related with symmetric spaces, while the
LY-algebras with $[x,y,z]=0$ are the Lie algebras.  Less known
examples  can be found in \cite{BDP} where a detailed analysis on
the algebraic structure of LY-algebras arising from homogeneous
spaces which are quotients of the compact Lie group $G_2$ is given.
\smallskip

These nonassociative binary-ternary algebras have been treated by
several authors in connection with geometric problems on homogeneous
spaces \cite{Kik79,Kik81,Sag65,Sag68,SagWin}, but no much
information on their algebraic structure is available yet.
\smallskip

Given a Lie-Yamaguti algebra $(\m,x\cdot y,[x,y,z])$ and any two
elements $x,y\in\m$, the linear map $D(x,y):\m\rightarrow\m$,
$z\mapsto D(x,y)(z)=[x,y,z]$ is, due to (LY5) and (LY6), a
derivation of both the binary and ternary products. These
derivations will be called \emph{inner derivations}. Moreover, let
$D(\m,\m)$ denote the linear span of the inner derivations. Then
$D(\m,\m)$ is closed under commutation thanks to (LY6). Consider the
vector space $\g(\m)=D(\m,\m)\oplus\m$, and endow it  with the
anticommutative multiplication given, for any $x,y,z,t\in \m$, by:
\begin{equation}\label{eq:gm}
\begin{split}
&[D(x,y),D(z,t)]= D([x,y,z],t)+D(z,[x,y,t]),\\
&[D(x,y),z]=D(x,y)(z)=[x,y,z],\\
&[z,t]=D(z,t)+z\cdot t.
\end{split}
\end{equation}
Note that the Lie algebra $D(\m,\m)$ becomes a subalgebra of
$\g(\m)$.

Then it is straightforward \cite{Yam58} to check that $\g(\m)$ is a
Lie algebra, called the \emph{standard enveloping Lie algebra} of
the Lie-Yamaguti algebra $\m$. The binary and ternary products in
$\m$ coincide with those given by \eqref{eq:binter}, where
$\h=D(\m,\m)$.

\smallskip

As was mentioned above, the Lie triple systems are precisely those
LY-algebras with trivial binary product. These correspond to the
symmetric homogeneous spaces. Following \cite[\S 16]{Nom54}, a
symmetric homogeneous space $G/H$ is said to be \emph{irreducible}
if the action of $\ad\h$ on $\m$ is irreducible, where
$\g=\h\oplus\m$ is the canonical decomposition of the Lie algebra
$\g$ of $G$.

This suggests the following definition:

\begin{definition}\label{df:irreducible}
A Lie-Yamaguti algebra $(\m, x\cdot y,[x,y,z])$ is said to be
\emph{irreducible} if $\m$ is an irreducible module for its Lie
algebra of inner derivations $D(\m,\m)$.
\end{definition}

Geometrically, the irreducible LY-algebras correspond to the
isotropy irreducible homogeneous spaces studied by Wolf in \cite{Wolf} ``as a first step toward understanding the geometry of the riemannian homogeneous spaces''. Likewise, the classification of the irreducible LY-algebras constitutes a first step in our understanding of this variety of algebras.
Concerning the isotropy irreducible homogeneous spaces, Wolf remarks that
``the results are surprising, for there are a  large number of
nonsymmetric isotropy irreducible coset spaces $G/K$, and only a few
examples had been known before.  One of the most interesting
class is $\mathbf{SO}(\dimens K)/\ad K$ for an arbitrary compact
simple Lie group $K$''. These spaces $\mathbf{SO}(\dimens K)/\ad K$ show a clear pattern, but there appear many more examples in the classification, where no such clear pattern appears.

Here it will be shown that most of the irreducible LY-algebras follow
clear patterns if several kinds of nonassociative algebraic systems
are used, not just Lie algebras. In fact, most of the irreducible
LY-algebras will be shown, here and in the forthcoming paper
\cite{forthcoming}, to appear inside simple Lie algebras as
orthogonal complements of subalgebras of derivations of Lie and
Jordan algebras, Freudenthal triple systems and Jordan pairs.

\medskip

Let us fix some notation to be used throughout this paper. All the
algebraic systems will be assumed to be finite dimensional over an
algebraically closed ground field $k$ of characteristic $0$.
Unadorned tensor products will be considered over this ground field
$k$. Given a Lie algebra $\g$ and a subalgebra $\h$, the pair
$(\g,\h)$ will be said to be a \emph{reductive pair} (see
\cite{Sag68}) if there is a complementary subspace $\m$ of $\h$ with
$[\h,\m]\subseteq \m$. The decomposition $\g=\h\oplus\m$ will then
be called a \emph{reductive decomposition} of the Lie algebra $\g$.

In particular, given a LY-algebra $(\m, x\cdot y,[x,y,z])$, the pair
$\bigl(\g(\m),D(\m,\m)\bigr)$ is a reductive pair.

The following result is instrumental:

\begin{proposition}\label{pr:envueltasimple}
Let $\g=\h\oplus \m$ be a reductive decomposition of a simple Lie
algebra $\g$, with $\m\ne 0$. Then $\g$ and $\h$ are isomorphic,
respectively, to the standard enveloping Lie algebra and the inner
derivation algebra of the Lie-Yamaguti algebra $(\m,x\cdot
y,[x,y,z])$ given by \eqref{eq:binter}. Moreover, in case $\h$ is
semisimple and $\m$ is irreducible as a module for $\h$, either $\h$
and $\m$ are isomorphic as $\ad \h$-modules or $\m=\h^\perp$, the
orthogonal complement of $\h$ relative to the Killing form of $\g$.
\end{proposition}

\begin{proof}
For the first assertion it is enough to note that $\pi_{\h}([\m,
\m])\oplus\m\, (=[\m,\m]+\m)$ and $\{x\in\h :[x,\m]=0\}$  are ideals
of $\g$. Hence, if $\g$ is simple, $\pi_{\h}([\m,\m])=\h$ holds, and
$\h$  embeds naturally in $D(\m,\m)\subseteq\End_k(\m)$. From here
it follows that the map $\g\to \g(\m)$  given by $h\in \h \mapsto
\ad h \mid_\m$ and $x\in \m \mapsto x$ is an isomorphism from $\g$
to $\g(\m)$ which sends $\h$ onto $D(\m,\m)$. Moreover, in case $\h$
is semisimple, $\h$ is anisotropic with respect to the Killing form
of $\g$ (by Cartan's criterion, as $\g$ is a faithful
$\ad_{\g}\h$-module), so $\g=\h\oplus \h^\perp$ and the orthogonal
projection, $\pi_\h(\m)$ from $\m$ onto $\h$ is an ideal of $\h$. By
irreducibility of $\m$, either $\pi_\h(\m)=0$ and therefore
$\m=\h^\perp$, or $\m$ is isomorphic to $\pi_\h(\m)$. In the latter
case, since the action of $\h$ on $\m$ is faithful, it follows that
$\h=\pi_\h(\m)$, as required.
\end{proof}

\bigskip

The paper is organized as follows.  Section 2 will be devoted to
establish the main structural features on Lie inner derivations and
standard enveloping Lie algebras of the irreducible LY-algebras.
These will be split into three non-overlapping types: adjoint,
non-simple and generic. The final result in this section shows that
LY-algebras of adjoint type are essentially simple Lie algebras. The
classification of the LY-algebras of non-simple type is the goal of
 the rest of the paper, while the generic type will be treated in a
forthcoming paper.  Section 3 will give examples of irreducible
LY-algebras, many of them appearing inside Lie algebras obtained
by means of the Tits construction of Lie algebras in \cite{Tits66}
in terms of composition algebras and suitable Jordan algebras. Then
in Section 4 these examples will be shown to exhaust the irreducible
LY-algebras of non-simple type.


\section{Irreducible Lie-Yamaguti algebras. Initial classification}

For irreducible LY-algebras $\m$, the irreducibility as a module for
$D(\m,\m)$, together with Schur's Lemma,  quickly leads to the
following result:

\begin{theorem}\label{th:estructura}
Let $(\m, x\cdot y,[x,y,z])$ be an irreducible  LY-algebra. Then
$D(\m,\m)$ is a semisimple and maximal subalgebra of the standard
enveloping Lie algebra $\g(\m)$. Moreover, $\g(\m)$ is simple in
case $\m$ and $D(\m,\m)$ are not isomorphic as $D(\m,\m)$-modules.
\end{theorem}

\begin{proof}
Any subalgebra $M$ of $\g(\m)$ containing $D(\m,\m)$ decomposes as
$M=D(\m,\m) \oplus (M \cap \m)$, thus $M=D(\m,\m)$ or $\g(\m)$ by
the irreducibility of $\m$. Hence $D(\m,\m)$ is a maximal
subalgebra. The irreducibility of $\m$ also implies  that $D(\m,\m)$
is a reductive algebra with $\dimens Z(D(\m,\m)) \le 1$ (see
\cite[Proposition 19.1]{Hum72}). If $Z(D(\m,\m))= Fz$, Schur's Lemma
shows that there is a scalar $\alpha\in k$ such that  $\ad_{\g(\m)}z
\mid_{\m}= \alpha Id$ holds. In this case, for any $x,y \in \m$ we
have
\begin{equation}\label{eq:thestruc1}
\ad_{\g(\m)}z([x,y])=2 \alpha [x,y]
\end{equation}
If $\alpha \neq 0$, since $2 \alpha$ is not an eingenvalue of
$\ad_{\g(\m)}z$, from \eqref{eq:thestruc1} it follows that
$[\m,\m]=0$, so $D(\m,\m)=0$, a contradiction. Hence $\alpha=0$
which implies $z=0$ because $\m$ is a faithful module for
$D(\m,\m)$, and therefore $D(\m,\m)$ is semisimple.

Finally, if $\m$ is not the adjoint module for $D(\m,\m)$, given a
proper ideal $I$ of $\g(\m)$, we have $I \cap \m =0$: otherwise, $I
\cap \m =\m$ and then $\g(\m)= \m + [\m,\m] \subseteq I$, a
contradiction. Hence $[I\cap D(\m,\m),\m]=0$ and therefore $I\cap
D(\m,\m)=0$. By maximality of $D(\m,\m)$, $\g(\m)$ can be decomposed
as
\begin{equation}\label{eq:thestruc2}
 \g(\m)=D(\m,\m) \oplus I=D(\m,\m) \oplus \m
\end{equation}
thus $I $ is isomorphic to $\m$ as $D(\m,\m)$-modules. From
\eqref{eq:thestruc2}, $\m \oplus I$ is a $D(\m,\m)$-module
isomorphic to $\m \oplus \m$ and it is easily checked that $P=(\m
\oplus I)\cap D(\m,\m)$ is a nonzero ideal of $D(\m,\m)$ isomorphic
to $\m$. So that $D(\m,\m)=P \oplus Q$ (direct sum of ideals). Now,
as $[P,Q]=0$ and $P$ is isomorphic to  $\m$ as $D(\m,\m)$-modules,
$[Q,\m]=0$ follows, and therefore, since $\m$ is a faithful module,
$Q=0$ and this contradicts the fact that $\m$ is not the adjoint
module for $D(\m,\m)$.
\end{proof}

The previous theorem points out two different situations depending
on the LY-algebra module behavior. This observation, together with
Proposition \ref {pr:envueltasimple}, leads to the following
definition and structure result:

\begin{definition}\label{df:LYadj}
A LY-algebra $\m$ is said to be of {\em adjoint type} if $\m$ is the
adjoint module for the inner derivation algebra $D(\m,\m)$.
\end{definition}

\begin{corollary}\label{co:nonadj}
The irreducible LY-algebras which are not of adjoint type are the
orthogonal subspaces of their inner derivation algebras relative to
the Killing form of their standard enveloping Lie algebras. In
particular, these irreducible LY-algebras are contragredient modules
for $D(\m,\m)$. \hfill $\square$
\end{corollary}
Note that Theorem \ref{th:estructura} guarantees the simplicity of
standard enveloping Lie algebras of the non-adjoint irreducible
LY-algebras. In the adjoint type, according to Theorem
\ref{th:adjunto} below, the standard enveloping Lie algebras are
never simple. So these results split the classification of
irreducible LY-algebras into the following non overlapping types:

\begin{equation}\label{tipos}
\begin{array}{ll}

\textsc{Adjoint Type:}&\textrm{$\m$ is the adjoint module for $D(\m,\m)$}\\

\textsc{Non-Simple Type:} &\textrm{$D(\m,\m)$ is not simple}\\

\textsc{Generic Type:}&\textrm{Both $\g(\m)$ and $D(\m,\m)$ are simple}
\end{array}
\end{equation}

Moreover, the complete classification of the first type is easily
obtained as we shall show in the sequel. The non-simple type will be
studied in Section 4, while the generic type will be the object of a
forthcoming paper \cite{forthcoming}.

Given any irreducible LY-algebra of adjoint type $(\m, x\cdot y,
[x,y,z])$, the inner derivation Lie algebra $D(\m,\m)$ is simple.
Thus from \cite{BenOs} the subspace
\begin{equation}\label{eq:benos}
\Hom_{D(\m,\m)}(\Lambda^2\m,\m)
\end{equation}
is one dimensional and spanned by the Lie bracket in
$D(\m,\m)$. So, given a $D(\m,\m)$-module isomorphism $\varphi:
D(\m,\m) \to \m$, the maps
\begin{equation}\label{eq:circ1}
\cdot:\m \times \m \to \m,\ (x,y) \mapsto x\cdot y
\end{equation}
and
\begin{equation}\label{eq:triple1}
\tilde D:\m \times \m \to \m,\ (x,y) \mapsto
\varphi(D(x,y))=\varphi([x,y,-])
\end{equation}
belong to the vector space in (\ref{eq:benos}), and hence there
exist scalars $\alpha, \beta \in k$, $\beta \ne 0$, such that
\begin{equation}\label{eq:circ2}
\varphi (x) \cdot \varphi (y)= \alpha \varphi ([x,y])
\end{equation}
\begin{equation}\label{eq:triple2}
 \tilde D(\varphi (x),\varphi (y))=\beta \varphi([x,y])
\end{equation}
for any $x,y \in D(\m,\m)$. Moreover, there is then an isomorphism
of Lie algebras:
\begin{equation}\label{eq:g(m)}
\g(\m)=D(\m,\m)\oplus\varphi(D(\m,\m))\cong K \otimes D(\m,\m),
\end{equation}
 where $K$ is the quotient
$k[t]/(t^2-\alpha t -\beta)$ of the polynomial ring on the variable
$t$, that maps $x+\varphi(y)$ to $1\otimes x + \bar t\otimes y$, for
any $x,y\in D(\m,\m)$, where $\bar t$ denotes the class of the
variable $t$ modulo the ideal $(t^2-\alpha t-\beta)$. Now, depending
on $\alpha$, two different situations appear:
\begin{itemize}
\item
 If $\alpha =0$, it can be assumed that $\beta=1$
(by taking $\frac{1}{\sqrt \beta} \varphi$ instead of $\varphi$). In
this case, $\m$ is a LY-algebra with trivial binary product, so a
Lie triple system, isomorphic to the triple system given by the Lie
algebra $D(\m,\m)$ with trivial binary product and ternary product
given by $[x,y,z]=[[x,y],z]$. In this case, $\g(\m)$ is the direct
sum of two copies of $D(\m,\m)$.
\item
 If $\alpha \ne 0$, it can be assumed that $\alpha =1$
(by taking $\frac{1}{\alpha} \varphi$ instead of $\varphi$). Then
$\m$ is isomorphic to the the LY-algebra $D(\m,\m)$ with binary and
ternary products given by $x \cdot y=[x,y]$ and
$[x,y,z]:=\beta[[x,y],z]$. Moreover, if $\beta \ne -1/4$
(equivalently, $K \cong k \times k$), $\g(\m)$ is the direct sum of
two copies of $D(\m,\m)$. In case $\beta = -1/4$, the enveloping Lie
algebra $\g(\m)$ is isomorphic to the Lie algebra $k[t]/(t^2)\otimes
D(\m,\m)$, whose solvable (actually abelian) radical is $(t)/(t^
2)\otimes D(\m,\m)$.
\end{itemize}

Now, from our previous discussion we obtain:

\begin{theorem}\label{th:adjunto}
Up to isomorphism, the LY-algebras of adjoint type are the simple
Lie algebras $L$ with binary and ternary products  of one of the
following types:
\begin{enumerate}
 \item [(i)] $x\cdot y=0$ and $[x,y,z]=[[x,y],z]$
\item [(ii)] $x \cdot y=[x,y]$ and $[x,y,z]=\beta[[x,y],z]$, $\beta \ne 0$
\end{enumerate}
where $[x,y]$ is the Lie bracket in $L$. Moreover, the standard
enveloping Lie algebra is a direct sum of two copies of the simple
Lie algebra $L$ in case \textup{(i)} or case \textup{(ii)} with $\beta \ne  -1/4$. In
case \textup{(ii)} with $\beta = -1/4$, the standard enveloping Lie algebra is
isomorphic to $k[t]/(t^ 2)\otimes L$. \hfill $\square$
\end{theorem}

\begin{remark}\label{re:adjoint}
This Theorem, together with Theorem \ref{th:estructura}, shows that
the adjoint type in \eqref{tipos} does not overlap with the other
two types, as the standard enveloping Lie algebra is never simple
for the adjoint type, while it is always simple in the non-simple
and generic types. \hfill $\square$
\end{remark}


\section{Examples of non-simple type irreducible
LY-algebras}\label{Section:Examples}

Several examples  of irreducible LY-algebras and of its enveloping
Lie algebras will be shown in this section. In the next section,
these examples will be proved to exhaust all the possibilities for
non-simple type irreducible LY-algebras.

\subsection{Classical examples}

Given a vector space $V$ and a nondegenerate $\epsilon$-symmetric
bilinear form $\varphi$ on $V$ (that is, $\varphi$ is symmetric if
$\epsilon=1$ and skew-symmetric if $\epsilon=-1$), consider the Lie
algebra $\anti(V,\varphi)=\{ f\in \gl(V):
\varphi(f(v),w)=-\varphi(v,f(w))\ \forall v,w\in V\}$ of skew
symmetric linear maps relative to $\varphi$. Thus,
$\anti(V,\varphi)=\frso(V,\varphi)$ (respectively
$\frsp(V,\varphi)$) if $\varphi$ is symmetric (respectively
skew-symmetric). This Lie algebra $\anti(V,\varphi)$ is spanned by
the linear maps $\varphi_{v,w}=\varphi(v,.)w-\epsilon\varphi(w,.)v$,
for $v,w\in V$. The bracket of two such linear maps is given by:
\begin{equation}\label{eq:bracketvarphis}
\begin{split}
[\varphi_{a,b},\varphi_{x,y}]&=
  \varphi_{\varphi_{a,b}(x),y}+\varphi_{x,\varphi_{a,b}(y)}\\
  &=\varphi(a,x)\varphi_{b,y}-\varphi(x,b)\varphi_{a,y}
   -\varphi(y,a)\varphi_{b,x}+\varphi(b,y)\varphi_{a,x},
\end{split}
\end{equation}
for any $a,b,x,y\in V$.

Moreover, the subspace $\sym(V,\varphi)=\{ f\in \End_k(V):
\varphi(f(v),w)=\varphi(v,f(w))\ \forall v,w\in V\}$ of the
symmetric linear maps relative to $\varphi$ is closed under the
symmetrized product:
\[
f\bullet g=\frac{1}{2}(fg+gf).
\]
($\sym(V,\varphi)$ is a special Jordan algebra.) Use will be made of
the subspace of trace zero symmetric linear maps, which will be
denoted by $\sym_0(V,\varphi)$. It is clear that
$\sym(V,\varphi)=k1_V\oplus \sym_0(V,\varphi)$, where $1_V$ denotes
the identity map on $V$.

\medskip

\begin{examples}\label{ex:skewV1plusV2}
Let $(V_i,\varphi_i)$, $i=1,2$, be two vector spaces endowed with
nondegenerate $\epsilon$-symmetric bilinear forms ($\epsilon=\pm
1$), with $1\leq \dim V_1\leq \dim V_2$. Consider the direct sum
$V_1\oplus V_2$ with the nondegenerate $\epsilon$-symmetric bilinear
form given by the orthogonal sum $\varphi=\varphi_1\perp\varphi_2$.
Then, under the natural identifications,
\[
\begin{split}
\anti(V_1\oplus V_2,\varphi)
 &=\bigl(\varphi_{V_1,V_1}\oplus\varphi_{V_2,V_2}\bigr)
   \oplus \varphi_{V_1,V_2}\\
 &=\bigl(\anti(V_1,\varphi_1)\oplus\anti(V_2,\varphi_2)\bigr)
   \oplus \varphi_{V_1,V_2}.
\end{split}
\]
This gives a $\Z_2$-grading of $\anti(V_1\oplus V_2,\varphi)$. As a
module for the even part
$\anti(V_1,\varphi_1)\oplus\anti(V_2,\varphi_2)$, the odd part
$\varphi_{V_1,V_2}$ is isomorphic to $V_1\otimes V_2$, and it is
irreducible unless $\epsilon=1$ and either $\dim V_1=1$ and $1\leq\dim V_2\leq 2$, or $\dim V_1=2$. The Lie bracket of two basic elements in
$\varphi_{V_1,V_2}$ is, due to \eqref{eq:bracketvarphis} and since
$V_1$ and $V_2$ are orthogonal, given by:
\[
[\varphi_{x_1,x_2},\varphi_{y_1,y_2}]
 =\varphi_2(x_2,y_2)(\varphi_1)_{x_1,y_1}+\varphi_1(x_1,y_1)(\varphi_2)_{x_2,y_2},
\]
for any $x_1,y_1\in V_1$ and $x_2,y_2\in V_2$.

Therefore, unless $\epsilon=1$ and either $\dim V_1=1$ and $1\leq\dim V_2\leq 2$, or
$\dim V_1=2$, $\m=V_1\otimes V_2$ is an irreducible LY-algebra
(actually an irreducible Lie triple system) with trivial binary
product, and ternary product given by (see \eqref{eq:binter}):
\begin{equation}\label{eq:terV1otimesV2}
\begin{split}
[x_1\otimes x_2, y_1\otimes y_2,z_1\otimes z_2]
 &=\varphi_2(x_2,y_2)\Bigl((\varphi_1)_{x_1,y_1}(z_1)\otimes
 z_2\Bigr)\\
  &\qquad\qquad +\varphi_1(x_1,y_1)\Bigl(z_1\otimes
  (\varphi_2)_{x_2,y_2}(z_2)\Bigr).
\end{split}
\end{equation}
\end{examples}

\begin{examples}\label{ex:skewV1otimesV2}
Let $(V_i,\varphi_i)$ be a vector space endowed with a nondegenerate
$\epsilon_i$-symmetric bilinear form ($i=1,2$), with $2\leq \dim
V_1\leq \dim V_2$. Then $V_1\otimes V_2$ is endowed with the
nondegenerate $\epsilon_1\epsilon_2$-symmetric bilinear form
$\varphi=\varphi_1\otimes \varphi_2$. For $i=1,2$, we have:
\[
\gl(V_i)=\anti(V_i,\varphi_i)\oplus\sym(V_i,\varphi_i)=
 \anti(V_i,\varphi_i)\oplus\sym_0(V_i,\varphi_i)\oplus k1_{V_i},
\]
and
\[
\begin{split}
\anti(&V_1\otimes  V_2,\varphi)\\
&=
 \Bigl(\anti(V_1,\varphi_1)\otimes  k1_{V_2}\, \oplus\,
 k1_{V_1}\otimes \anti(V_2,\varphi_2)\Bigr)\oplus\\
 &\qquad\Bigl(\bigl(\anti(V_1,\varphi_1)\otimes
 \sym_0(V_2,\varphi_2)\bigr)
 \oplus \bigl(\sym_0(V_1,\varphi_1)\otimes \anti(V_2,\varphi_2)\bigr)\Bigr).
\end{split}
\]
This provides a reductive decomposition $\g=\h\oplus\m$ of
$\g=\anti(V_1\otimes  V_2,\varphi)$, where $\h\simeq
\anti(V_1,\varphi_1)\oplus\anti(V_2,\varphi_2)$ and
$\m=\bigl(\anti(V_1,\varphi_1)\otimes \sym_0(V_2,\varphi_2)\bigr)
 \oplus \bigl(\sym_0(V_1,\varphi_1)\otimes \anti(V_2,\varphi_2)\bigr)$.

In this situation, if $\m$ is an irreducible module for $\h$, then
$\dim V_1=2$ and $\epsilon_1=-1$ (which forces
$\sym_0(V_1,\varphi_1)$ to be trivial).

Assuming $\dim V_1=2$, $\epsilon_1=-1$, and $\dim V_2=n\geq 2$, then
$\m=\frsp(V_1,\varphi_1)\otimes \sym_0(V_2,\varphi_2)$ is an
irreducible module for $\h$ if and only if either $\epsilon_2=-1$
and $\dim V_2= 2m\geq 4$, or $\epsilon_2=1$ and $\dim V_2\geq 3$.

With these assumptions, for $a,b\in \frsp(V_1,\varphi_1)$ and
$f,g\in \sym_0(V_2,\varphi_2)$, $ab+ba=\tr(ab)1_{V_1}$ (as
$\frsp(V_1,\varphi_1)$ is isomorphic to the Lie algebra
$\frsl_2(k)$), and hence
$ab=\frac{1}{2}\bigl([a,b]+\tr(ab)1_{V_1}\bigr)$ and $ba=
\frac{1}{2}\bigl(-[a,b]+\tr(ab)1_{V_1}\bigr)$ hold. Moreover, if the
dimension of $V_2$ is $n$, then for any $f,g\in
\sym_0(V_2,\varphi_2)$, the element
$fg+gf-\frac{2}{n}\tr(fg)1_{V_2}$ also belongs to
$\sym_0(V_2,\varphi_2)$.

Now, for any $a,b\in \frsp(V_1,\varphi_1)$ and $f,g\in
\sym_0(V_2,\varphi_2)$:
\begin{equation}\label{eq:skewV1otimesV2}
\begin{split}
[a\otimes f,b\otimes g]&=
   ab\otimes fg -ba\otimes gf\\
   &=\frac{1}{2}[a,b]\otimes(fg+gf) +
   \frac{1}{2}\tr(ab)1_{V_1}\otimes [f,g]\\
   &=\Bigl([a,b]\otimes \frac{1}{n}\tr(fg)1_{V_2}
   +\frac{1}{2}\tr(ab)1_{V_1}\otimes [f,g]\Bigr) \\
   &\qquad\qquad +
   \frac{1}{2}[a,b]\otimes\bigl(fg+gf-\frac{2}{n}\tr(fg)1_{V_2}\bigr).
\end{split}
\end{equation}
Therefore, the binary and ternary products in the irreducible
LY-algebra $\m=\anti(V_1,\varphi_1)\otimes \sym_0(V_2,\varphi_2)$
are given by:
\begin{equation}\label{eq:binterskewV1otimesV2}
\begin{split}
(a\otimes f)\cdot(b\otimes
g)&=\frac{1}{2}[a,b]\otimes\bigl(fg+gf-\frac{2}{n}\tr(fg)1_{V_2}\bigr),\\[4pt]
[a\otimes f,b\otimes g,c\otimes h]&=
  \frac{1}{n}\tr(fg)[[a,b],c]\otimes h  +
  \frac{1}{2}\tr(ab)c\otimes [[f,g],h],
\end{split}
\end{equation}
for any $a,b,c\in \anti(V_1,\varphi_1)=\frsl(V_1)$ and $f,g,h\in
\sym_0(V_2,\varphi_2)$.

Note that for $\epsilon_2=-1$ and $\dim V_2=4$, it is easily checked that $[[a,b],c]=2\tr(bc)a-2\tr(ac)b$ for any $a,b,c\in\frsl(V_1)$, while $fg+gf-\frac{1}{2}\tr(fg)1_{V_2}=0$ and $[[f,g],h]=\tr(gh)f-\tr(fh)g$ for any $f,g,h\in\sym_0(V_2,\varphi_2)$. Hence \eqref{eq:binterskewV1otimesV2} becomes in this case
\[
\begin{split}
(a\otimes f)\cdot(b\otimes
g)&=0,\\[4pt]
[a\otimes f,b\otimes g,c\otimes h]&=
  \frac{1}{2}\tr(fg)\bigl(\tr(bc)a-\tr(ac)b\bigr)\otimes h \\
  &\qquad\qquad +
  \frac{1}{2}\tr(ab)c\otimes \bigl(\tr(gh)f-\tr(fh)g\bigr),
\end{split}
\]
for any $a,b,c\in \anti(V_1,\varphi_1)=\frsl(V_1)$ and $f,g,h\in
\sym_0(V_2,\varphi_2)$, and thus the triple product coincides with the expression in \eqref{eq:terV1otimesV2} for $\varphi_1(a,b)=\tr(ab)$ and $\varphi_2(f,g)=-\frac{1}{2}\tr(fg)$. Therefore, the irreducible Lie-Yamaguti algebras obtained here for $\dim V_1=2$, $\dim V_2=4$ and $\epsilon_1=-1=\epsilon_2$ coincides with the one obtained in Example \ref{ex:skewV1plusV2} for two vector spaces of dimension $3$ and $5$.
\hfill\qed

\end{examples}

\bigskip

\begin{examples}\label{ex:slV1otimesslV2}
Let now $V_1$ and $V_2$ be two vector spaces with $2\leq \dim
V_1\leq\dim V_2$. The algebra of endomorphisms of the tensor product
$V_1\otimes V_2$ can be identified with the tensor product of the
algebras of endomorphisms of $V_1$ and $V_2$. Moreover, the general
Lie algebra $\gl(V_i)$ decomposes as $\gl(V_i)=k1_{V_i}\oplus
\frsl(V_i)$. Then
\[
\begin{split}
\frsl(V_1\otimes V_2)&=\bigl(\frsl(V_1)\otimes k1_{V_2}\bigr)\oplus
\bigl(k1_{V_1}\otimes \frsl(V_2)\bigr)\oplus \bigl(\frsl(V_1)\otimes
\frsl(V_2)\bigr)\\
 &\simeq \bigl(\frsl(V_1)\oplus\frsl(V_2)\bigr)\oplus
 \bigl(\frsl(V_1)\otimes \frsl(V_2)\bigr)
\end{split}
\]
gives a reductive decomposition, and this shows that
$\m=\frsl(V_1)\otimes \frsl(V_2)$ is an irreducible LY-algebra. For
$a,b\in \frsl(V_1)$, both $[a,b]=ab-ba$ and
$ab+ba-\frac{2}{n_1}\tr(ab)1_{V_1}$ belong to $\frsl(V_1)$, where
$n_i$ denotes the dimension of $V_i$, $i=1,2$. Therefore, for any
$a,b\in \frsl(V_1)$ and $f,g\in \frsl(V_2)$:
\begin{equation}\label{eq:slV1slV2}
\begin{split}
[a\otimes f,b\otimes g]&=ab\otimes fg-ba\otimes gf\\
 &=\Bigl([a,b]\otimes\frac{1}{n_2}\tr(fg)1_{V_2} +
  \frac{1}{n_1}\tr(ab)1_{V_1}\otimes [f,g]\Bigr) \\
  &\qquad\quad +\Bigl(\frac{1}{2}[a,b]\otimes
  (fg+gf-\frac{2}{n_2}\tr(fg)1_{V_2})\\
  &\qquad\qquad\quad +
  (ab+ba-\frac{2}{n_1}\tr(ab)1_{V_1})\otimes\frac{1}{2}[f,g]\Bigr).
\end{split}
\end{equation}
Hence, the binary and the ternary products in the irreducible
LY-algebra $\m=\frsl(V_1)\otimes\frsl(V_2)$ are given by:
\begin{equation}\label{eq:binterslV1slV2}
\begin{split}
(a\otimes f)\cdot(b\otimes g)&=\frac{1}{2}[a,b]\otimes
(fg+gf-\frac{2}{n_2}\tr(fg)1_{V_2}) \\
&\qquad +
(ab+ba-\frac{2}{n_1}\tr(ab)1_{V_1})\otimes\frac{1}{2}[f,g],\\[8pt]
[a\otimes f,b\otimes g,c\otimes h]&=
  [[a,b],c]\otimes\frac{1}{n_2}\tr(fg)h+\frac{1}{n_1}\tr(ab)c\otimes
  [[f,g],h],
\end{split}
\end{equation}
for any $a,b,c\in \frsl(V_1)$ and $f,g,h\in \frsl(V_2)$.

Note that, as noted in Example \ref{ex:skewV1otimesV2}, if $\dim V_1=2$, then for any $a,b,c\in \frsl(V_1)$, $ab+ba-\tr(ab)1_{V_1}=0$, while $[[a,b],c]=2\tr(bc)a-2\tr(ac)b$. Hence, if $\dim V_1=\dim V_2=2$, \eqref{eq:binterslV1slV2} becomes:
\[
\begin{split}
(a\otimes f)\cdot(b\otimes g)&=0,\\[4pt]
[a\otimes f,b\otimes g,c\otimes h]&=\tr(fg)\bigl(\tr(bc)a-\tr(ac)b\bigr)\otimes h \\
  &\qquad\qquad +
  \tr(ab)c\otimes \bigl(\tr(gh)f-\tr(fh)g\bigr),
\end{split}
\]
for any $a,b,c\in \frsl(V_1)$ and $f,g,h\in \frsl(V_2)$, and thus the triple product coincides with the expression in \eqref{eq:terV1otimesV2} for $\varphi_1(a,b)=\tr(ab)$ and $\varphi_2(f,g)=-\tr(fg)$. Therefore, the irreducible Lie-Yamaguti algebras obtained here for $\dim V_1=2=\dim V_2$  coincides with the one obtained in Example \ref{ex:skewV1plusV2} for two vector spaces of dimension $3$. \hfill\qed

\end{examples}

\bigskip

\subsection{Generalized Tits Construction}

Examples \ref{ex:skewV1plusV2} and \ref{ex:skewV1otimesV2} can be
seen as instances of a Generalized Tits Construction, due to Benkart
and Zelmanov \cite{BenZel}, which will now be reviewed in a way
suitable for our purposes.

\smallskip

Let $X$ be a unital $k$-algebra endowed with a \emph{normalized}
trace $t:X\rightarrow k$. This means that $t$ is a linear map with
$t(1)=1$, $t(xy)=t(yx)$ and $t((xy)z)=t(x(yz))$ for any $x,y,z\in
X$. Then $X=k1\oplus X_0$, where $X_0=\{x\in X: t(x)=0\}$ is the set
of trace zero elements in $X$. For $x,y\in X_0$, the element
$x*y=xy-t(xy)1$ lies in $X_0$ too, and this defines a bilinear
multiplication on $X_0$. Assume there is a skew-symmetric bilinear
transformation $D:X_0\times X_0\rightarrow \Der(X)$, where $\Der(X)$
denotes the Lie algebra of derivations of $X$, such that $D_{x,y}$
leaves invariant $X_0$ and $[E,D_{x,y}]=D_{E(x),y}+D_{x,E(y)}$, for
any $x,y\in X_0$ and $E\in D_{X_0,X_0}$. Here $D_{X_0,X_0}$ denotes the
Lie subalgebra of $\Der(X)$ spanned by the image of the map $D$.

An easy example of this situation is given by the Jordan algebras of
symmetric bilinear forms: let $V$ be a vector space endowed with a
symmetric bilinear form $\varphi$, then $\cJ(V,\varphi)=k1\oplus V$,
with commutative multiplication given by
\[
(\alpha 1+v)(\beta 1+w)=\bigl(\alpha\beta +\varphi(v,w)\bigr)1
+\bigl(\alpha w+\beta v\bigr),
\]
for any $\alpha,\beta\in k$ and $v,w\in V$. Here the normalized
trace is given by $t(1)=1$ and $t(v)=0$ for any $v\in V$, while the
skew symmetric map $D$ is given by $D(v,w)=\varphi_{v,w}$ for any
$v,w\in V$.

Let $Y=k1\oplus Y_0$ be another such algebra, with normalized trace
also denoted by $t$, multiplication on $Y_0$ denoted by $\star$ and
analogous skew-symmetric bilinear map $d:Y_0\times Y_0\rightarrow
\Der(Y)$. Then the vector space
\begin{equation}\label{eq:TXY}
\cT(X,Y)=D_{X_0,X_0}\oplus \bigl(X_0\otimes Y_0\bigr)\oplus
d_{Y_0,Y_0}
\end{equation}
is an anticommutative algebra with multiplication defined by
\begin{equation}\label{eq:bracketonTXY}
\begin{split}
&\text{$D_{X_0,X_0}$ and $d_{Y_0,Y_0}$ are subalgebras of
$\cT(X,Y)$},\\
&[D_{X_0,X_0},d_{Y_0,Y_0}]=0,\\
&[D,x\otimes y]=D(x)\otimes y,\\
&[d,x\otimes y]=x\otimes d(y),\\
&[x\otimes y,x'\otimes y']=t(yy')D_{x,x'} + (x*x')\otimes (y\star
y')+ t(xx')d_{y,y'},
\end{split}
\end{equation}
for any $x,x'\in X_0$, $y,y'\in Y_0$, $D\in D_{X_0,X_0}$ and $d\in
d_{Y_0,Y_0}$.

\begin{proposition}\label{pr:TXYLie} \textup{(\cite[Proposition
3.9]{BenZel})}
The algebra $\cT(X,Y)$ above is a Lie algebra provided the following
relations hold
\begin{equation*}
\begin{split}
\textup{(i)\ }&\  \displaystyle{\sum_{\circlearrowleft}
 t\bigl((x_{1}*x_{2}) x_{3}\bigr)\,
d_{y_1 \star y_2, y_3}}=0,\\[6pt]
\textup{(ii)\ }&\  \displaystyle{\sum_{\circlearrowleft}
 t\bigl( (y_1 \star y_2) y_{3}\bigr)
\,D_{x_1* x_2,x_3}}=0,\\[6pt]
\textup{(iii)\ }&\ \displaystyle{\sum_{\circlearrowleft}
 \Bigl(D_{x_1,x_2}(x_3) \otimes t\bigl(y_1 y_2\bigr) y_3}
   + (x_1*x_2)*x_3 \otimes (y_1 \star y_2)\star y_3\\[-6pt]
  &\qquad\qquad\qquad\qquad +
t(x_1  x_2) x_3\otimes d_{y_1, y_2}(y_3)\Bigr)=0
\end{split}
\end{equation*}
for any $x_1,x_2,x_{3} \in X_0$ and any $y_1,y_2,y_3 \in Y_0$. The
notation ``$\displaystyle{\sum_\circlearrowleft}$'' indicates
summation over the cyclic permutation of the indices.
\end{proposition}

Note that, in case $\cT(X,Y)$ is a Lie algebra, then $X_0\otimes
Y_0$ becomes a LY-algebra with binary and ternary products given by
\begin{equation}\label{eq:binterX0otimesY0}
\begin{split}
(x_1\otimes y_1)\cdot (x_2\otimes y_2)&=(x_1*x_2)\otimes (y_1\star
y_2),\\
[x_1\otimes y_1,x_2\otimes y_2,x_3\otimes
y_3]&=D_{x_1,x_2}(x_3)\otimes t(y_1y_2)y_3\\
&\qquad\qquad +t(x_1x_2)x_3\otimes d_{y_1,y_2}(y_3),
\end{split}
\end{equation}
for any $x_1,x_2,x_3\in X_0$ and $y_1,y_2,y_3\in Y_0$. This will be
called the \emph{Lie-Yamaguti algebra inside $\cT(X,Y)$}.

\medskip

\begin{remark}\label{re:TJVJW}
An important example where $\cT(X,Y)$ is a Lie algebra arises when
Jordan algebras of symmetric bilinear forms are used as the
ingredients \cite[3.28]{BenZel}. If $(V_1,\varphi_1)$ and
$(V_2,\varphi_2)$ are two vector spaces endowed with nondegenerate
symmetric bilinear forms and $\cJ_1=\cJ(V_1,\varphi_1)$ and
$\cJ_2=\cJ(V_2,\varphi_2)$ are the corresponding Jordan algebras,
then
$D_{(\cJ_i)_0,(\cJ_i)_0}=\frso(V_i,\varphi_i)=\anti(V_i,\varphi_i)$,
$i=1,2$, and the reductive decomposition
\[
\begin{split}
\cT(\cJ_1,\cJ_2)&=\Bigl(D_{(\cJ_1)_0,(\cJ_1)_0}\oplus
D_{(\cJ_2)_0,(\cJ_2)_0}\Bigr)\oplus \bigl((\cJ_1)_0\otimes (\cJ_2)_0\bigr)\\
 &\simeq \bigl(\frso(V_1,\varphi_1)\oplus\frso(V_2,\varphi_2)\Bigr)
  \oplus \bigl(V_1\otimes V_2\bigr)
\end{split}
\]
coincides, with the natural identifications, with the reductive
decomposition in Example \ref{ex:skewV1plusV2} with $\epsilon=1$.
Therefore, the LY-algebras in Example \ref{ex:skewV1plusV2} with
$\epsilon=1$, are the LY-algebras obtained inside the Generalized
Tits Construction $\cT(\cJ_1,\cJ_2)$, where $\cJ_1$ and $\cJ_2$ are
Jordan algebras of nondegenerate symmetric bilinear forms.

\smallskip

Moreover, the Generalized Tits Construction $\cT(X,Y)$ can be
assumed to be associated with algebras $(X_0,*)$ and $(Y_0,\star)$
having  skew-symmetric bilinear forms, and with symmetric maps $D$
and $d$ (see \cite[3.33]{BenZel}). In particular, it works when
$\cJ_i=k1\oplus V_i$ is the Jordan superalgebra of a nondegenerate
skew-symmetric bilinear form $\varphi_i$, $i=1,2$. Here the even
part of the superalgebra $\cJ_i$ is just $k1$, while the odd part is
$V_i$. With exactly the same arguments as above, it is checked that
the LY-algebras in Example \ref{ex:skewV1plusV2} with $\epsilon=-1$,
are exactly the LY-algebras obtained inside the Generalized Tits
Construction $\cT(\cJ_1,\cJ_2)$, where $\cJ_1$ and $\cJ_2$ are
Jordan superalgebras of nondegenerate skew-symmetric bilinear forms.
\hfill\qed
\end{remark}

\medskip

But the Generalized Tits Construction has its origin in the Classical Tits
Construction in \cite{Tits66}, which is the source of further
examples of LY-algebras.

\begin{examples}\label{ex:ClassicalTits} (\textbf{Classical Tits
Construction})

Let $\cC$ be a unital composition algebra with norm $n$ (see
\cite{Jac58}). Thus, $\cC$ is a finite dimensional unital
$k$-algebra, with the nondegenerate quadratic form $n:\cC\rightarrow
k$ such that $n(ab)=n(a)n(b)$ for any $a,b\in \cC$. Then, each
element satisfies the degree $2$ equation
\begin{equation}\label{eq:deg2}
a^2-\tr(a)a+n(a)1=0,
\end{equation}
where $\tr(a)=n(a,1)\,\bigl(=n(a+1)-n(a)-n(1)\bigr)$ is called the
\emph{trace}. The subspace of trace zero elements will be denoted by
$\cC_0$. The algebra $\cC$ is endowed of a canonical involution,
given by $\bar x= \tr(x)1-x$.

Moreover, for any $a,b\in \cC$, the linear map
$D_{a,b}:\cC\rightarrow \cC$ given by
\begin{equation}\label{eq:Dab}
D_{a,b}(c)=\frac{1}{4}\Bigl([[a,b],c]+3(a,c,b)\Bigr)
\end{equation}
where $[a,b]=ab-ba$ is the commutator, and $(a,c,b)=(ac)b-a(cb)$ the
associator, is a derivation: the \emph{inner derivation} determined
by the elements $a,b$ (see \cite[Chapter III, \S 8]{Sch}). These
derivations span the whole Lie algebra of derivations $\Der(\cC)$.
Moreover, they satisfy
\begin{equation}\label{eq:Dcyclic}
D_{a,b}=-D_{b,a},\quad D_{ab,c}+D_{bc,a}+D_{ca,b}=0,
\end{equation}
for any $a,b,c\in \cC$. The normalized trace here is
$t=\frac{1}{2}\tr$, and the multiplication $*$ on $\cC_0$ is just
$a*b=ab-t(ab)1=\frac{1}{2}[a,b]$, since $ab+ba=\tr(ab)1$, for any
$a,b\in \cC_0$.

The only unital composition algebras (recall that the ground field
is being assumed to be algebraically closed) are, up to isomorphism,
the ground field $k$, the cartesian product of two copies of the
ground field $\cK=k\times k$, the split quaternion algebra, which is
the algebra of two by two matrices $\cQ=\Mat_2(k)$, and the split
octonion algebra $\cO$ (see, for instance, \cite[Chapter 2]{ZSSS}).

\smallskip

On the other hand, given a finite dimensional unital Jordan algebra
$\J$ of degree $n$ (see \cite{JacJA}), we denote by $T(x)$ its {\em
generic trace} ($T(1)=n$), by $N(x)$ its  {\em generic norm} and by
$\J_0$ the subspace of trace zero elements. Then $t=\frac{1}{n}T$ is
a normalized trace. If $R_x$ is the right multiplication by $x$, the
map $d_{x,y}:\J \to \J$ given by
\begin{equation}\label{eq:dxy}
d_{x,y}(z)=[R_x,R_y]
\end{equation}
is a derivation.

Now, given a unital composition algebra $\cC$, one may consider the
subspace $H_n(\cC)$ of $n\times n$ hermitian matrices over $\cC$
with respect to the standard involution $(x_{ij})^*=({\bar
x}_{ji})$. This is a Jordan algebra with the symmetrized product
$x\bullet y=\frac{1}{2}(xy+yx)$ if either $\cC$ is associative or $n\leq 3$. For $\cC=k$, this is just the
algebra of symmetric $n\times n$ matrices, for $\cC=\cK$ this is
isomorphic to the algebra $\Mat_n(k)$ with the symmetrized product,
while for $\cC=\cQ$ this is the algebra of symmetric matrices for
the symplectic involution in $\Mat_n(\Mat_2(k))\simeq \Mat_{2n}(k)$.

Up to isomorphisms, the simple Jordan algebras are the following:
\begin{description}
\settowidth{\labelwidth}{XXX}%
\setlength{\leftmargin}{50pt}
\item[degree $1$] The ground field $k$.
\item[degree $2$] The Jordan algebras of nondegenerate symmetric
bilinear forms $\cJ(V,\varphi)$.
\item[degree $n\geq 3$] The Jordan algebras $H_n(k)$, $H_n(\cK)$ and
$H_n(\cQ)$, plus the degree three Jordan algebra $H_3(\cO)$.
\end{description}

For the simple Jordan algebras, the derivations $d_{x,y}$'s span the
whole Lie algebra of derivations $\Der(\cJ)$.

It turns out that the conditions in Proposition \ref{pr:TXYLie} are
satisfied if $X=\cC$ is a unital composition algebra and $Y=\cJ$ is
a degree three Jordan algebra (see \cite{Tits66} and
\cite[Proposition 3.24]{BenZel}). This is the Classical Tits
Construction, which gives rise to Freudenthal's Magic Square (Table
\ref{ta:FMS}), if the simple Jordan algebras of degree three are
taken as the second ingredient.

\begin{table}[h!]
$$
\vbox{\offinterlineskip
 \halign{\hfil\ $#$\ \hfil&%
 \vreglon #%
 &\hfil\ $#$\ \hfil&\hfil\ $#$\ \hfil
 &\hfil\ $#$\ \hfil&\hfil\ $#$\ \hfil\cr
 \bigstrut \cT(\cC,\cJ)&&H_3(k)&H_3(\cK)&H_3(\cQ)&H_3(\cO)\cr
 \multispan6{\hreglonfill}\cr
 k&&A_1&A_2&C_3&F_4\cr
 \bigstrut \cK&& A_2&A_2\oplus A_2&A_5&E_6\cr
 \bigstrut \cQ&&C_3 & A_5&D_6&E_7\cr
 \bigstrut \cO&& F_4& E_6& E_7&E_8\cr}}
$$
\medskip \caption{Freudenthal's Magic Square}\label{ta:FMS}
\end{table}

In the third and fourth rows of this Magic Square (that is, if the
composition algebras $\cQ$ and $\cO$ are considered), there appears
the reductive decomposition:
\[
\cT(\cC,\cJ)=\Bigl(\Der(\cC)\oplus\Der(\cJ)\Bigr)\oplus
\bigl(\cC_0\otimes \cJ_0\bigr),
\]
and this shows that, with $\dim\cC$ being either $4$ or $8$ and
$\cJ$ being a simple degree three Jordan algebra,
$\cC_0\otimes\cJ_0$ is an irreducible LY-algebra with binary and
ternary products given by
\begin{equation}\label{eq:binterClassicalTits}
\begin{split}
(a\otimes x)\cdot(b\otimes y)&=\frac{1}{2}[a,b]\otimes (x\bullet y-t(x\bullet
y)1),\\[4pt]
[a_1\otimes x_1,a_2\otimes x_2,a_3\otimes x_3]&=
  D_{a_1,a_2}(a_3)\otimes t(x_1\bullet x_2)x_3\\
  &\qquad\qquad + t(a_1a_2)a_3\otimes
  d_{x_1,x_2}(x_3)
\end{split}
\end{equation}
for any $a_1,a_2,a_3\in \cC$ and $x_1,x_2,x_3\in \cJ$.\hfill\qed

\end{examples}

Consider the third row of the Classical Tits Construction, with an
arbitrary unital Jordan algebra of degree $n$. Since $\cQ$ is
associative, the inner derivation $D_{a,b}$ in \eqref{eq:Dab} is
just $\frac{1}{4}\ad_{[a,b]}$, thus $\Der(\cQ)$ can be identified to
$\cQ_0$. The linear map $\bigl(\cQ_0\otimes
\cJ)\oplus\Der(\cJ)\rightarrow \cT(\cQ,\cJ)$, which is the identity
on $\Der(\cJ)$ and takes $a\otimes 1$ to $\ad_a\in\Der(\cQ)$ and
$a\otimes x$ to $2(a\otimes x)$, for any $a\in \cQ_0$ and
$x\in\cJ_0$, is then a bijection. Under this bijection, the
anticommutative product on $\cT(\cQ,\cJ)$ is transferred to the
following product on
$\g=\bigl(\cQ_0\otimes\cJ\bigr)\oplus\Der(\cJ)$:
\begin{equation}\label{eq:bracketonTKK}
\begin{split}
&\text{$\Der(\cJ)$ is a subalgebra of
$\g$},\\
&[d,a\otimes x]=a\otimes d(x),\\
&[a\otimes x,b\otimes y]=([a,b]\otimes x\bullet y)+2\tr(ab)d_{x,y}
\end{split}
\end{equation}
for any $a,b\in \cQ_0$, $x,y\in\cJ$ and $d\in \Der(\cJ)$.

For any Jordan algebra $\cJ$, Tits showed in \cite{Tits62} that this
bracket gives a Lie algebra $\g$. This is the well-known
Tits-Kantor-Koecher Lie algebra attached to $\cJ$ (see
\cite{Tits62,Kan,Ko67}). Therefore, the third row of the Classical
Tits Construction is valid for any unital Jordan algebra, not just
for degree three Jordan algebras.

\begin{remark}\label{re:TQJ}
Take, for instance, the Jordan algebra $\cJ=H_n(\cK)$, which
can be identified with the algebra of $n\times n$ matrices
$\Mat_n(k)$, but with the Jordan product $x\bullet
y=\frac{1}{2}(xy+yx)=\frac{1}{2}(l_x+r_x)(y)$, where $l_x$ and $r_x$
denote, respectively, the left and right multiplication in the
associative algebra $\Mat_n(k)$. Then for any $x,y\in \cJ$, the
inner derivation $d_{x,y}$ equals
$\frac{1}{4}[l_x+r_x,l_y+r_y]=\frac{1}{4}\ad_{[x,y]}$. Since
$\cQ=\Mat_2(k)$, for any $a,b\in \cQ_0$ and $x,y\in \cJ_0$, the Lie
bracket in \eqref{eq:bracketonTKK} gives, for any $a,b\in
\cQ_0=\frsl_2(k)$ and $x,y\in\cJ_0=\frsl_n(k)$:
\[
[a\otimes x,b\otimes y]
 =\frac{1}{n}\tr(xy)[a,b]+\frac{1}{2}[a,b]\otimes
(xy+yx-\frac{2}{n}\tr(xy)1)+\frac{1}{2}\tr(ab)[x,y].
\]
This is exactly the multiplication in \eqref{eq:slV1slV2} with
$n_1=2$ and $n_2=n$.

Actually, we can think of the construction in Example
\ref{ex:slV1otimesslV2} as a sort of Generalized Tits Construction
$\cT(H_{n_1}(\cK),H_{n_2}(\cK))$.

\smallskip

On the other hand, let $(V_2,\varphi_2)$ be a vector space endowed
with a nondegenerate $\epsilon$-symmetric bilinear form. Then
$\cJ=\sym(V_2,\varphi_2)$ is a Jordan algebra with the symmetrized
product $f\bullet g=\frac{1}{2}(fg+gf)$. If $\epsilon=1$ and $\dim
W=n$, then $\cJ$ is isomorphic to $H_n(k)$, while if $\epsilon=-1$
and $\dim W=2n$, then $\cJ$ is isomorphic to $H_n(\cQ)$. As in the
previous remark, and since $\cQ_0=\frsl_2(k)\simeq
\frsp(V_1,\varphi_1)$, where $V_1$ is a two-dimensional vector space
endowed with a nonzero skew-symmetric bilinear form $\varphi_1$, the
Lie bracket in \eqref{eq:bracketonTKK} is exactly the multiplication
in \eqref{eq:skewV1otimesV2}. This means that the irreducible
LY-algebra in Example \ref{ex:skewV1otimesV2} is the LY-algebra
obtained inside $\cT(\cQ,\sym(V_2,\varphi_2))$.

\smallskip

Finally, if again $(V_2,\varphi_2)$ is a vector space endowed with a
nondegenerate symmetric bilinear form and $\cJ_2=\cJ(V_2,\varphi_2)$
is the associated Jordan algebra, since $\ad_{\cQ_0}$ is isomorphic
to the orthogonal Lie algebra $\frso(\cQ_0,n\vert_{\cQ_0})$ (recall
that $n$ denotes the norm of the composition algebra $\cQ$, which in
this case coincides with the determinant of $2\times 2$ matrices),
it follows easily that $\cT(\cQ,\cJ_2)$ is isomorphic to
$\cT(\cJ_1,\cJ_2)$ (see Remark \ref{re:TJVJW}), where $\cJ_1$ is the
Jordan algebra of the nondegenerate symmetric bilinear form
$n\vert_{\cQ_0}$.

Therefore, concerning the LY-algebras inside the Classical Tits
Construction, only the cases $\cT(\cQ,H_3(\cO))$ and
$\cT(\cO,H_3(\cC))$ for $\cC=k$, $\cK$, $\cQ$, or $\cO$ are not
covered by the previous examples. \hfill\qed
\end{remark}

\bigskip

\subsection{Symplectic triple systems}

There is another type of examples of irreducible LY-algebras
(actually, of irreducible Lie triple systems) with exceptional
enveloping Lie algebra, which appears in terms of the so called
\emph{symplectic triple systems} or, equivalently, of Freudenthal
triple systems.

Symplectic triple systems were introduced first in \cite{YamAs}.
They are basic ingredients in the construction of some $5$-graded
Lie algebras (and hence $\Z_2$-graded algebras). They consist of a
vector space $\cT$ endowed with a trilinear product $\{xyz\}$ and a
nonzero skew-symmetric bilinear form $(x,y)$ satisfying some
conditions (see Definition 2.1 in \cite{Eld06} for a complete
description). Following \cite{Eld06}, from any symplectic triple
system $\cT$, a Lie algebra can be defined on the vector space
\begin{equation}\label{eq:gsymplectic}
\g(\cT)=\frsp(V)\oplus \bigl(V\otimes \cT\bigr) \oplus \Inder(\cT)
\end{equation}
where $V$ is a $2$-dimensional space endowed with a nonzero
skew-symmetric bilinear form $\varphi$ and $\Inder \cT=\eespan
\langle d_{x,y}=\{xy\cdot\}: x,y \in \cT\rangle$ is the Lie algebra
of inner derivations of $T$, by considering the anticommutative product
given by:

\begin{itemize}
\item $\frsp(V)$ and $\Inder(\cT)$ are Lie subalgebras of $\g(\cT)$,
\item $[\frsp(V),\Inder(\cT)]=0$,
\item $[f+d,v\otimes x]=f(v)\otimes x+v\otimes d(x)$,
\item with $\varphi_{u,v}=\varphi(u,.)v+\varphi(v,.)u$ (as usual),
\begin{equation}\label{productosimplectico}
[u\otimes x,v\otimes y]=(x,y)\varphi_{u,v}+\varphi(u,v)d_{x,y}\end{equation}
\end{itemize}
for all $f\in\frsp(V)$, $d\in\Inder(\cT)$, $u,v\in V$ and $x,y\in
\cT$. The decomposition $\g_{\bar0}=\frsp(V) \oplus \Inder(\cT)$ and
$\g_{\bar 1}= V\otimes \cT$ provides a $\Z_2$-graduation on
$\g(\cT)$, so the odd part  $\g_{\bar 1}= V\otimes \cT$ is a
LY-algebra with trivial binary product (Lie triple system). The
simplicity of $\g(\cT)$ is equivalent to that of $\cT$, which is
characterized by the nondegeneracy of the associated bilinear form
$(x,y)$. Note that viewing  $\frsp(V)$ as $\frsl(V)$, and $V$ as its
natural module, a $5$-grading is obtained by looking at the
eigenspaces of the adjoint action of a Cartan subalgebra in
$\frsl(V)$. This feature relates symplectic triples with
structurable algebras with a one-dimensional space of skew-hermitian
elements (see \cite{AF84}).

Symplectic triple systems are also related to Freudenthal triple
systems (see \cite{Mey68}) and to Faulkner ternary algebras
introduced in \cite{Fau71,FauFe}. In fact, in the simple case all
these systems are essentially equivalent (see \cite{Eld06}).

Among the simple symplectic triple systems (see \cite{Eld06}) use
will be made of the following ones:
\begin{equation}\label{eq:TJsymplectic}
\cT_{\J}=\Bigl\{\begin{pmatrix} \alpha & a \\ b &
 \beta \end{pmatrix}: \alpha, \beta \in k, a, b \in \J\Bigr\}
\end{equation}
where $\J=\J ordan(n,c)$ is the Jordan algebra of a nondegenerate
cubic form $n$ with basepoint (see \cite[II.4.3]{McC04} for a
definition) of one of the following types: $\J=k,
n(\alpha)=\alpha^3$ and $t(\alpha, \beta)=3\alpha \beta$ or
$\J=H_3(\cC)$ for a unital composition algebra $\cC$. Theorem 2.21
in \cite{Eld06} displays carefully the product and bilinear form for
the  triple systems $\cT_\J$ and Theorem 2.30 describes the
structure of $\g(\cT_\J)$. The information on the Lie algebras
involved is given in Table \ref{ta:symplecticsquare}.

\medskip

\begin{table}[h!]
\begin{center}
\begin{tabular}{|c|c|c|c|c|c|}
\hline
&&&&&\\
\raisebox{1.5ex}[0pt]{$\cT_\J$}
 & \raisebox{1.5ex}[0pt]{$\cT_k$}
  &\raisebox{1.5ex}[0pt]{$\cT_{H_3(k)}$}
  & \raisebox{1.5ex}[0pt]{$\cT_{H_3(\cK)}$}
  &\raisebox{1.5ex}[0pt]{$\cT_{H_3(\cQ)}$}
  &\raisebox{1.5ex}[0pt]{$\cT_{H_3(\cO)}$}\\
\hline
&&&&&\\
\raisebox{1.5ex}[0pt]{$\Inder \cT_\J$}
 & \raisebox{1.5ex}[0pt]{$A_1$}
 & \raisebox{1.5ex}[0pt]{$C_3$}
 & \raisebox{1.5ex}[0pt]{$A_5$}
 &  \raisebox{1.5ex}[0pt]{$D_6$}
 &\raisebox{1.5ex}[0pt]{$E_7$}\\
\hline
\hline
&&&&&\\
\raisebox{1.5ex}[0pt]{$\g(\cT_\J)$}
 &\raisebox{1.5ex}[0pt]{$G_2$}
 &\raisebox{1.5ex}[0pt]{$F_4$}
 & \raisebox{1.5ex}[0pt]{$E_6$}
 & \raisebox{1.5ex}[0pt]{$E_7$}
 & \raisebox{1.5ex}[0pt]{$E_8$}\\
\hline
\end{tabular}
\end{center}
\smallskip
\caption{$\g(\cT_\J)$-algebras}\label{ta:symplecticsquare}
\end{table}

From these symplectic triple systems, five new constructions of
exceptional Lie algebras, exactly one for each simple Jordan algebra
$\cJ$ above, and hence a new family of LY-algebras appears:

\begin{examples}\label{LY:symplectic}
Let $\cT_\J$ be the symplectic triple system defined in
\eqref{eq:TJsymplectic} where either  $\J$ is $k$ with norm
$n(\alpha)=\alpha^3$, or it is $H_3(\cC)$ with its generic norm for
a unital composition algebra $\cC$. The Lie algebra $\g(\cT_\J)$
given in \eqref{eq:gsymplectic} is simple and presents the reductive
decomposition $\g(\cT_\J)=\h\oplus \m$, where $\h=\frsp(V) \oplus
\Inder \cT_\J$ and $\m=V\otimes \cT_\J$. In these cases, $\h$ is
isomorphic to the semisimple Lie algebra of type $A_1\oplus L$, with
$L=A_1$, $C_3$, $A_5$, $D_6$ or $E_7$ as in Table
\ref{ta:symplecticsquare}. Moreover, $\h$ acts irreducible on $\m$
and therefore $V\otimes \cT_\cJ$ becomes an irreducible LY-algebra
with trivial binary  product (that is, it is an irreducible Lie
triple system) and ternary product given by:
\begin{equation}\label{productotriplesimplectico}
[u\otimes x, v\otimes y, w\otimes z]
 =(x,y)\varphi_{u,v}(w)\otimes z+\varphi(u,v)w\otimes \{xyz\}
\end{equation}
where $(x,y)$ and $\{xyz\}$ are the alternating form and the triple
product of $\cT_\J$. Its standard enveloping Lie algebra is, because
of Proposition \ref{pr:envueltasimple}, the Lie algebra
$\g(\cT_\cJ)$, whose type is given in Table
\ref{ta:symplecticsquare} too. \hfill $\square$

\end{examples}

\bigskip


\section{Classification}

As shown in Section 2, the irreducible Lie-Yamaguti
algebras of non-simple type are those for which the inner derivation
algebra is semisimple and nonsimple. According to Theorem
\ref{th:estructura}, the standard enveloping Lie algebras of such
LY-algebras are simple Lie algebras, so following Proposition
\ref{pr:envueltasimple} the classification of such LY-algebras can
be reduced to determine the reductive decompositions $\g=\h\oplus
\m$ satisfying
\begin{equation}\label{condicionescasonosimpleI}
\begin{array}{rl}
&  \mathrm{(a)}\quad
\textrm{$\g$ is a simple Lie algebra}\\
&\textrm{(b)}\quad
\textrm{$\h$ is a semisimple and non simple subalgebra of $\g$} \\
& \mathrm{(c)}\quad
\textrm{$\m$ is an irreducible $\ad\h$-module}
\end{array}
\end{equation}

In this section we classify the irreducible LY-algebras of
non-simple type and, first of all, the irreducible LY-algebras whose
standard enveloping is classical, that is, isomorphic to either
$\frsl_n(k)$ (special), $n \ge 2$, $\so_n(k)$ (orthogonal), $n \ge
3$, or $\frsp_{2n}(k)$ (symplectic), $n \ge 1$.

\begin{theorem}\label{th:irreduciblesclasicas}
Let  $(\m, x\cdot y,[x,y,z])$ be an irreducible LY-algebra of non-simple type whose standard enveloping Lie algebra is  simple and classical. Then, up
to isomorphism, either:

\begin{enumerate}

\item[(i)]
$\m=\frsl(V_1)\otimes \frsl(V_2)$
for some vector spaces $V_1$ and $V_2$ with $2\leq\dim V_1\leq \dim
V_2$ and $(\dim V_1,\dim V_2)\ne (2,2)$, as in Example \ref{ex:slV1otimesslV2}, with binary and ternary
products given in \eqref{eq:binterslV1slV2}. \\
In this case the standard enveloping Lie algebra is isomorphic to
the special linear algebra $\frsl(V_1\otimes V_2)$ and the inner
derivation algebra to $\frsl(V_1)\oplus\frsl(V_2)$.

\item[(ii)]
$\m=V_1\otimes V_2$ for some vector spaces $V_1$ and $V_2$ endowed with
nondegenerate symmetric bilinear forms with $3\leq \dim V_1\leq \dim
V_2$ as in Example \ref{ex:skewV1plusV2}. This is an irreducible Lie
triple system, whose triple product is given in
\eqref{eq:terV1otimesV2}. Alternatively, this is the LY-algebra
inside the Tits construction $\cT(\cJ(V_1),\cJ(V_2))$ for two Jordan
algebras of symmetric bilinear forms in Remark \ref{re:TJVJW}.
\\
In this case the standard enveloping Lie algebra is isomorphic to
the orthogonal Lie algebra $\frso(V_1\oplus V_2)$ and the inner
derivation algebra to $\frso(V_1)\oplus\frso(V_2)$.

\item[(iii)]
$\m=V_1\otimes V_2$ for some vector spaces $V_1$ and $V_2$ endowed
with nondegenerate skew-symmetric bilinear forms with $2\leq \dim
V_1\leq \dim V_2$ as in Example \ref{ex:skewV1plusV2}. This is an
irreducible Lie triple system, whose triple product is given in
\eqref{eq:terV1otimesV2}. Alternatively, this is the LY-algebra
inside the Tits construction $\cT(\cJ(V_1),\cJ(V_2))$ for two Jordan
superalgebras of skew-symmetric bilinear forms in Remark
\ref{re:TJVJW}.
\\
In this case the standard enveloping Lie algebra is isomorphic to
the symplectic Lie algebra $\frsp(V_1\oplus V_2)$ and the inner
derivation algebra to $\frsp(V_1)\oplus\frsp(V_2)$.

\item[(iv)]
$\m=\frsp(V_1)\otimes\cJ_0$, where $V_1$ is a
two-dimensional vector space endowed with a nonzero skew-symmetric
bilinear form and $\cJ$ is the Jordan algebra $H_n(k)$ for $n\geq 3$
(that is, isomorphic to $\sym(V_2,\varphi_2)$, for a vector space
$V_2$ of dimension $n$ endowed with a nondegenerate symmetric
bilinear form $\varphi_2$). The binary and ternary products are
given in \eqref{eq:binterskewV1otimesV2}. Alternatively, this is the
LY-algebra inside the Tits construction $\cT(\cQ,H_n(k))$ (see
Remark \ref{re:TQJ}).
\\
In this case the standard enveloping Lie algebra is isomorphic to
the symplectic Lie algebra $\frsp(V_1\otimes V_2)\simeq
\frsp_{2n}(k)$, and the inner derivation algebra to
$\frsp(V_1)\oplus\frso(V_2)$.

\item[(v)]
$\m=\frsp(V_1)\otimes\cJ_0$, where $V_1$ is a
two-dimensional vector space endowed with a nonzero skew-symmetric
bilinear form and $\cJ$ is the Jordan algebra $H_n(\cQ)$ for
$n\geq 3$ (that is, isomorphic to $\sym(V_2,\varphi_2)$, for a
vector space $V_2$ of dimension $2n$ endowed with a nondegenerate
skew-symmetric bilinear form $\varphi_2$). The binary and ternary
products are given in \eqref{eq:binterskewV1otimesV2}.
Alternatively, this is the LY-algebra inside the Tits construction
$\cT(\cQ,H_n(\cQ))$ (see Remark \ref{re:TQJ}).
\\
In this case the standard enveloping Lie algebra is isomorphic to
the orthogonal Lie algebra $\frso(V_1\otimes V_2)\simeq
\frso_{4n}(k)$, and the inner derivation algebra to
$\frsp(V_1)\oplus\frsp(V_2)$.

\end{enumerate}
\end{theorem}

\begin{proof}
The irreducible LY-algebras of non-simple type with classical
enveloping Lie algebras are those obtained from reductive
decompositions $\g=\h \oplus \m$ satisfying
\eqref{condicionescasonosimpleI}, where $\g$ is a classical simple
Lie algebra and $\h=\h_1\oplus \h_2$, $0\ne\h_i$ semisimple. In this
case, $\h$ is a maximal subalgebra of $\g$ and Proposition
\ref{pr:envueltasimple} asserts that $\m$ is exactly the orthogonal
complement of $\h$ with respect to the Killing form of $\g$.

Suppose first that $\g$ is (isomorphic to) the special linear Lie
algebra $\frsl(V)$ for some vector space $V$ of dimension $\geq 2$.
If $V$ were not irreducible as a module for $\h$, then by Weyl's
Theorem, there would exist $\h$-invariant subspaces $V_1$ and $V_2$
with $V=V_1\oplus V_2$, but then $\h$ would be contained in the
subalgebra  $\frsl(V_1)\oplus\frsl(V_2)$ which is not maximal.
Therefore, $V$ is irreducible too as a module for $\h$. Hence, up to
isomorphism, the $\h$-module $V$ decomposes as a tensor product
$V=V_1\otimes V_2$ for some irreducible module $V_1$ for $\h_1$ and
some irreducible module $V_2$ for $\h_2$. It can be assumed that
$2\leq \dim V_1\leq \dim V_2$. Then $\h$ is contained in the
subalgebra $\frsl(V_1)\otimes k1_{V_2}\oplus k1_{V_1}\otimes
\frsl(V_2)$ of $\frsl(V_1\otimes V_2)$ and, by maximality, $\h$ is
exactly this subalgebra. Hence, we are in the situation of Example
\ref{ex:slV1otimesslV2} and Proposition \ref{pr:envueltasimple}
shows that the only complementary subspace to $\h$ in $\g$ which is
$\h$-invariant is its orthogonal complement relative to the Killing
form. This uniqueness shows that we are dealing with the irreducible
LY-algebra in Example \ref{ex:slV1otimesslV2}, thus obtaining case
(i).

\medskip

Suppose now that $\g$ is isomorphic to the Lie algebra of skew
symmetric linear maps of a vector space $V$ endowed with a
nondegenerate symmetric or skew-symmetric bilinear map $\varphi$.

If $V$ is not irreducible as a module for $\h$, and $W$ is an
irreducible $\h$-submodule of $V$ with $\varphi(W,W)\ne 0$, then by
irreducibility the restriction of $\varphi$ to $W$ is nondegenerate,
so $V$ is the orthogonal sum $V=W\oplus W^{\perp}$. By maximality of
$\h$, $\h$ is precisely the subalgebra
$\anti(W)\oplus\anti(W^{\perp})$, and the situation of Example
\ref{ex:skewV1plusV2} appears. Because of the uniqueness in
Proposition \ref{pr:envueltasimple}, items (ii) (for symmetric
$\varphi$) or (iii) (for skew-symmetric $\varphi$) are obtained.

On the other hand, if $V$ is not irreducible as a module for $\h$,
and the restriction of $\varphi$ to any irreducible $\h$-submodule
of $V$ is trivial then, by Weyl's theorem on complete reducibility, given an irreducible submodule $W_1$, there is another irreducible submodule $W_2$ with $\varphi(W_1,W_2)\ne 0$. Since $\varphi(W_1,W_1)=0=\varphi(W_2,W_2)$, $W_1$ and $W_2$ are contragredient modules and $V=(W_1\oplus W_2)\oplus (W_1\oplus W_2)^\perp$. Proceeding in the same way with $(W_1\oplus W_2)^\perp$, it is obtained that $V=V_1\oplus V_2$ for
some $\h$-invariant subspaces $V_1$ and $V_2$ such that the
restrictions of $\varphi$ to $V_1$ and $V_2$ are trivial. Then $\h$
is contained in $\{f\in \anti(V,\varphi): f(V_i)\subseteq V_i,\,
i=1,2\}$, which is $\varphi_{V_1,V_2}$. But this contradicts the
maximality of $\h$, since $\varphi_{V_1,V_2}$ is contained in the
subalgebra $\varphi_{V_1,V_2}\oplus\varphi_{V_1,V_1}$.

Finally, if $V$ remains irreducible as a module for $\h$ then, as
above, there is a decomposition $V=V_1\otimes V_2$ for an
irreducible module $V_i$ for $\h_i$, $i=1,2$, endowed with a
nondegenerate symmetric or skew-symmetric bilinear form $\varphi_i$
such that $\varphi=\varphi_1\otimes\varphi_2$. By maximality of $\h$
and Proposition \ref{pr:envueltasimple}, we are in the situation of
Example \ref{ex:skewV1otimesV2}, thus obtaining cases (iv) and (v)
depending on $\varphi$ being either skew-symmetric or symmetric
respectively.
\end{proof}

\medskip

Now it is time to deal with the irreducible LY-algebras with
exceptional standard enveloping Lie algebras. These algebras appear
inside reductive decompositions $\g=\h\oplus \m$ satisfying
\eqref{condicionescasonosimpleI} with $\g$ a simple exceptional Lie
algebra, and hence of type $G_2$, $F_4$, $E_6$, $E_7$ or  $E_8$.
Over the complex field, a thorough description of the maximal
semisimple subalgebras  of the simple exceptional Lie algebras is
given in \cite{Dyn}. The following result shows that the reductive
decomposition we are looking for  can be transferred to the complex
field, so the results in \cite{Dyn} can be used over our ground
field to get the classification of the exceptional irreducible
LY-algebras of non-simple type.

\begin{lemma}\label{le:reduccionacomplejos}
Let $\g=\h\oplus\m$ be a reductive decomposition over our ground
field $k$. Then there is an algebraically closed subfield $k'$ of
$k$, an embedding $\iota: k'\rightarrow \C$ and a Lie algebra $\g'$
over $k'$ with a reductive decomposition $\g'=\h'\oplus \m'$ such
that $\g=\g'\otimes_{k'}k$, $\h=\h'\otimes_{k'}k$ and
$\m=\m'\otimes_{k'}k$.
\end{lemma}
\begin{proof}
Let $\{x_i: i=1,\ldots,n\}$ be a basis of $\g$ over $k$ such that
$\{x_i: i=1,\ldots,m\}$ is a basis of $\h$ and
$\{x_{m+1},\ldots,x_n\}$ is a basis of $\m$ ($1<m<n$). For any
$1\leq i\leq j\leq n$, $[x_i,x_j]=\sum_{i=1}^n\alpha_{ij}^kx_k$, for
some $\alpha_{ij}^k\in k$ (the structure constants). Note that the
decomposition being reductive means that $\alpha_{ij}^k=0$ for
$1\leq i\leq j\leq m$ and $m+1\leq k\leq n$ ($\h$ is a subalgebra),
and for $1\leq i,k\leq m$ and $m+1\leq j\leq n$. Let $k''$ be the
subfield of $k$ generated (over the rational numbers) by the
structure constants. Since the  transcendence degree of the
extension $\C/\Q$ is infinite, there is an embedding
$\iota'':k''\rightarrow \C$. Finally, let $k'$ be the algebraic
closure of $k''$ on $k$. By uniqueness of the algebraic closure,
$\iota''$ extends to an embedding $\iota:k'\rightarrow \C$. Now, it
is enough to take $\h'=\sum_{i=1}^m k'x_i$,
$\m'=\sum_{i=m+1}^nk'x_i$ and $\g'=\h'\oplus \m'$.
\end{proof}

Therefore, if $\g=\h\oplus\m$ is a reductive decomposition of a
simple exceptional Lie algebra over our ground field $k$, with $\h$
semisimple but not simple, and with $\m$ an irreducible module for
$\h$, take $\g'$, $\h'$ and $\m'$ as in the previous Lemma \ref{le:reduccionacomplejos}. Then there exists
the reductive decomposition $\tilde\g=\tilde\h\oplus\tilde\m$ over
$\C$, where $\tilde\g=\g'\otimes_{k'}\C$ (via $\iota$) and also
$\tilde\h=\h'\otimes_{k'}\C$ and $\tilde\m=\m'\otimes_{k'}\C$. Since
$\g$ is simple and $\g'$ is a form of $\g$, $\g'$ is simple too and
of the same type as $\g$, and hence so is $\tilde \g$. In the same
vein, $\h$, $\h'$ and $\tilde\h$ are semisimple Lie algebras of the
same type, and the highest weights of $\m$ and $\tilde\m$
``coincide'', as both are obtained from the highest weight of $\m'$
relative to a Cartan subalgebra and an ordering of the roots for
$\h'$.

The displayed list of maximal subalgebras of complex semisimple Lie
algebras given in \cite{Dyn} distinguishes the regular maximal
subalgebras and the so called $S$-subalgebras. Following \cite{Dyn},
a subalgebra $\rs$ of a semisimple Lie algebra $\g$ is said to be
{\em regular} in case  $\rs$ has a basis formed by some elements of
a Cartan subalgebra of $\g$ and some elements of its root spaces. On
the other hand, an {\em $S$-subalgebra} is a subalgebra $\s$ not
contained in any regular subalgebra. We observe that maximal
subalgebras are either regular or $S$-subalgebras and regular
maximal subalgebras have maximal rank, that is, the rank of the
semisimple algebras they are living in. Hence, the inner derivation
Lie algebras of the irreducible LY-algebras belong to one of these
classes of subalgebras and, in case of nonzero binary product, they
are necessarily $S$-subalgebras:

\begin{lemma}\label{Ssubalgebras}
Let $\m$ be an irreducible LY-algebra which is not of adjoint type.
If the binary product in $\m$ is not trivial, then the inner
derivation Lie algebra $D(\m,\m)$ is a maximal semisimple {\em
$S$-subalgebra} of the simple standard enveloping Lie algebra of
$\m$.
\end{lemma}
\begin{proof}
Following Theorem \ref{th:estructura} and Corollary \ref{co:nonadj},
$D(\m,\m)$ is a maximal semisimple subalgebra of the simple
enveloping Lie algebra $\g(\m)$ and $\m$ is a selfdual
$D(\m,\m)$-module. Let $\lambda$ be the highest weight of $\m$ as a
module for $D(\m,\m)$ with respect to a Cartan subalgebra $H$ of
$D(\m,\m)$ and an ordering of the roots, so $\m=V(\lambda)$ as a
module. Then $-\lambda$ is its lowest weight ($\m$ is self dual).
Since the binary product on $\m$ is nonzero, so is the vector space
$\Hom_{D(\m,\m)}(V(\lambda)\otimes V(\lambda), V(\lambda))$.
Moreover, any map $\varphi$ in this space is determined by
$\varphi(v_\lambda \otimes v_{-\lambda})\in V(\lambda)_0$ with
$v_\lambda$ and $ v_{-\lambda}$ weight vectors of weights $\lambda$
and $-\lambda$, and $V(\lambda)_0$ the zero weight space in
$V(\lambda)$. Then  $V(\lambda)_0 $ must be non trivial and, as
$V(\lambda)_0 $ is contained in the centralizer of $H$ in $\g(\m)$,
the subalgebra $H$ is not a Cartan subalgebra of $\g(\m)$.
Therefore, $D(\m,\m)$ is not a maximal rank subalgebra of $\g(\m)$
and hence it is an $S$-subalgebra.
\end{proof}

The irreducible LY-algebras of non-simple type whose standard
enveloping Lie algebra is exceptional are classified in the next
result.

\begin{theorem}\label{th:irreduciblesexcepcionales}
Let $(\m, x\cdot y,[x,y,z])$ be an irreducible LY-algebra of
non-simple type whose standard enveloping Lie algebra is a simple
exceptional Lie algebra. Then, up to isomorphism, either:

\begin{enumerate}

\item[(i)]
$\m=V\otimes \cT_\J$, where $V$ is a two dimensional vector space
endowed with a nonzero skew-symmetric bilinear form and $\cT_\J$ is
the symplectic triple system associated to a Jordan algebra $\J$
isomorphic  either to $k$, $H_3(k)$, $H_3(\cK)$, $H_3(\cQ)$ or
$H_3(\cO)$, as in Example \ref{LY:symplectic}. This is an
irreducible Lie triple system whose ternary product is given in
\eqref{productotriplesimplectico}.
\\
In this case, the standard enveloping Lie algebra is the exceptional
simple Lie algebra of type $G_2$ for $\cJ=k$, $F_4$ for
$\cJ=H_3(k)$, $E_6$ for $\cJ=H_3(\cK)$, $E_7$ for $\cJ=H_3(\cQ)$ and
$E_8$ for $\cJ=H_3(\cO)$, while its inner derivation Lie algebra is
isomorphic respectively to $\spl_2(k)\oplus \spl_2(k)$,
$\spl_2(k)\oplus \frsp_6(k)$, $\spl_2(k)\oplus \spl_6(k)$,
$\spl_2(k)\oplus \so_{12}(k)$ and $\spl_2(k)\oplus E_7$.

\item[(ii)]
$\m=\cO_0\otimes \J_0$, where $\cJ$ is one of the Jordan algebras
$H_3(k)$, $H_3(\cK)$, $H_3(\cQ)$ or $H_3(\cO)$. This is the
irreducible LY-algebra inside the Classical Tits Construction
$\cT(\cO,\J)$ in Example \ref{ex:ClassicalTits}. The binary and
ternary products are given in \eqref{eq:binterClassicalTits}.
\\
In this case, the standard enveloping Lie algebra is the exceptional
simple Lie algebra of type $F_4$ for $\cJ=H_3(k)$, $E_6$ for
$\cJ=H_3(\cK)$, $E_7$ for $\cJ=H_3(\cQ)$ and $E_8$ for
$\cJ=H_3(\cO)$, while its inner derivation Lie algebra is isomorphic
respectively to $G_2\oplus \spl_2(k)$, $G_2\oplus \spl_3(k)$,
$G_2\oplus \frsp_6(k)$ and $G_2\oplus F_4$.

\item[(iii)]
$\m=\cQ_0\otimes H_3(\cO)_0$  is the irreducible LY-algebra inside
the Classical Tits Construction $\cT(\cQ,H_3(\cO))$ in Example
\ref{ex:ClassicalTits}. The binary and ternary products are given in
\eqref{eq:binterClassicalTits}.
\\
In this case, the standard enveloping Lie algebra is the exceptional
simple Lie algebra of type $E_7$, while its inner derivation Lie
algebra is isomorphic $\spl_2(k)\oplus F_4$.

\end{enumerate}

\end{theorem}

\begin{proof}
Following (\ref{condicionescasonosimpleI}), we must find reductive
decompositions $\g=\h\oplus\m$ with $\g$ exceptional simple, $\h$
semisimple but not simple and $\m$ irreducible. In case the binary
product is trivial, $\m$ is an irreducible Lie triple system. Up to
isomorphism, these triple systems fit into one of the following
$(\g(\m), D(\m,\m), \m)$ possibilities (see \cite{Fau80}):
$(G_2,A_1\times A_1, V(\lambda_1)\otimes V(3\mu_1)), (F_4,A_1\times
C_3, V(\lambda_1)\otimes V(\mu_3)), (E_6,A_1\times A_5,
V(\lambda_1)\otimes V(\mu_3)), (E_7,A_1\times D_6,
V(\lambda_1)\otimes V(\mu_6)), (E_8,A_1\times E_7,
V(\lambda_1)\otimes V(\mu_7))$. In the above list, $V(\lambda)
\otimes V(\mu)$ indicates the irreducible module structure of $\m$,
described by means of the fundamental weights $\lambda_i$ and $
\mu_i$ relative to  fixed Cartan subalgebras in each component of
$\h=L_1\times L_2$. The notation follows \cite{Hum72}. In all these
cases, $\g$ is a $\Z_2$-graded simple Lie algebra in which the odd
part contains a $3$-dimensional simple ideal of type $A_1$ for which
the even part is a sum of copies of a 2-dimensional irreducible
module. Identifying $A_1$ and $V(\lambda_1)$ with $\frsp(V)$ and $V$
respectively, for a two dimensional vector space $V$ endowed with a
nonzero skew-symmetric bilinear form,  the following general
description for these reductive decompositions follows:
\begin{equation}\label{eq:envueltasimplecticos}
\g= \frsp(V)\oplus \s \oplus (V\otimes \cT)
\end{equation}
where $\s$ is a simple Lie algebra. Then, Theorem 2.9 in
\cite{Eld06} shows that $\cT$ is endowed with a structure of a
simple symplectic triple system obtained from the Lie bracket of $\g$
for which $\s=\Inder(\cT)$. It follows that $\g$ is the Lie algebra
$\g(\cT)$ in \eqref{eq:gsymplectic}. An inspection of the
classification of the simple symplectic triple systems displayed in
\cite[Theorem 2.30]{Eld06} shows that the only possibilities for
$\cT$ are those given in Example \ref{LY:symplectic}. Thus item (i)
is obtained.

Now let us assume that the binary product is not trivial. From Lemma
\ref {Ssubalgebras}, it follows that $\h$ is a maximal semisimple
$S$-subalgebra of $\g$. Because of \cite[Theorem 14.1]{Dyn}, there
exist only eight possible pairs $(\g,\h)$ with $\h$ not simple and
$\g$ exceptional: $(F_4, G_2\oplus A_1)$, $(E_6, G_2\oplus A_2)$,
$(E_7, G_2\oplus C_3)$, $(E_7, F_4\oplus A_1)$, $(E_7, G_2\oplus
A_1)$, $(E_7, A_1\oplus A_1)$, $(E_8, G_2\oplus F_4)$, $(E_8,
A_2\oplus A_1)$. Now, the irreducible and nontrivial action of $\h$
on $\m$ implies that this is a tensor product $\m=V(\lambda)\otimes
V(\mu)$ with $V(\lambda)$, $V(\mu)$ irreducible modules of nonzero
dominant weights $\lambda$ and $\mu$ for each one of the simple
components in $\h$. Computing dimensions and possible irreducible
modules of the involved algebras, the following  descriptions of
$\m$, as a module for $\h$ are obtained:
\begin{description}
\settowidth{\labelwidth}{XX}%
\setlength{\leftmargin}{30pt}
\item[$(F_4,G_2\oplus A_1)$]
Here $\dim \m=52-(14+3)=35=7\times 5$. The only possibility for $\m$
is to be the tensor product of the seven dimensional irreducible
module for $G_2$ and the five dimensional irreducible module for
$A_1$: $\m=V(\lambda_1)\otimes V(4\mu_1)$.

\item[$(E_6,G_2\oplus A_2)$]
Here $\dim\m=78-(14+8)=56$. The only possibility for $\m$ is to be
the tensor product of the seven dimensional irreducible module for
$G_2$ and the adjoint module for $A_2$: $\m=V(\lambda_1)\otimes
V(\mu_1+\mu_2)$.

\item[$(E_7,G_2\oplus C_3)$]
Here $\dim\m=133-(14+21)=98$. The only possibility for $\m$ is to be
the tensor product of the seven dimensional irreducible module for
$G_2$ and a fourteen dimensional module for $C_3$:
$\m=V(\lambda_1)\otimes V(\mu_2)$. (The weight $\mu_3$ for $C_3$
cannot occur as this module is not self dual.)

\item[$(E_7,F_4\oplus A_1)$]
Here $\dim\m=133-(52+3)=78$. The only possibility for $\m$ is to be
the tensor product of the twenty six dimensional irreducible module
for $F_4$ and the adjoint module for $A_1$: $\m=V(\lambda_4)\otimes
V(2\mu_1)$.

\item[$(E_8,G_2\oplus F_4)$]
Here $\dim\m=248-(14+52)=182$. The only possibility for $\m$ is to
be the tensor product of the seven dimensional irreducible module
for $G_2$ and the twenty six dimensional module for $F_4$:
$\m=V(\lambda_1)\otimes V(\mu_4)$.

\item[$(E_7, G_2\oplus A_1)$]
Here $\dim\m=133-(14+3)=116=2^2\times 29$. As $G_2$ has no
irreducible modules of dimension $2$, $4$, $29$ or $58$, this case
is not possible.

\item[$(E_7, A_1\oplus A_1)$]
Here $\dim\m=133-(3+3)=127$. Since $127$ is prime, there is no
possible factorization.

\item[$(E_8, A_2\oplus A_1)$]
Here $\dim\m=248-(8+3)=237=3\times 79$. As $A_2$ has no irreducible
module of dimension $79$ and its modules of dimension $3$ are not
selfdual, this case is impossible too.

\end{description}
 Note that the possible reductive decompositions above
fit exactly into the Classical Tits Construction of exceptional Lie
algebras given in Example \ref{ex:ClassicalTits}. By identifying
$G_2$ with  $\Der(\cO)$ and $V(\lambda_1)$ with $\cO_0$, and $F_4$
with $\Der(H_3(\cO))$ and  $V(\lambda_4)$ with $H_3(\cO)_0$, the
case $(E_8,G_2\oplus F_4)$ corresponds to $\cT(\cO,H_3(\cO))$. Also,
with the identifications $A_1\simeq\Der H_3(k)$ and $V(4\mu_1)\simeq
H_3(k)_0$, $A_2\simeq\Der H_3(\cK)$ and $V(\mu_1+\mu_2)\simeq
H_3(\cK)_0$ (recall $\cK=k\times k$), $C_3\simeq \Der H_3(\cQ)$ and $V(\mu_2)\simeq
H_3(\cQ)_0$, the cases $(F_4, G_2\oplus A_1)$, $(E_6, G_2\oplus
A_2)$ and $(E_7, G_2\oplus C_3)$ are given by $\cT(\cO,\J)$ with
$\J=H_3(k)$, $H_3(\cK)$ or $H_3(\cQ)$. Finally, the case
$(E_7,F_4\oplus A_1)$ corresponds to $\cT(\cQ,H_3(\cO))$ under the
identifications $F_4\simeq \Der H_3(\cO)$ and $V(\lambda_4)\simeq
H_3(\cO)_0$, $A_1\simeq\Der \cQ$ and $V(2\mu_1)\simeq \cQ_0$.

On the other hand, if $\cA$ denotes either the algebra of
quaternions or octonions,  the subspaces $\Hom_{\Der
\cA}(\cA_0\otimes \cA_0, \Der \cA)$, $\Hom_{\Der\cA}(\cA_0\otimes
\cA_0, k)$ and $\Hom_{\Der \cA}(\cA_0\otimes \cA_0, \cA_0)$ are
spanned by  $a\otimes b\mapsto D_{a,b}$,
$a\otimes b\mapsto \tr(ab)$ and $a\otimes b\mapsto [a,b]$ respectively, where $D_{a,b}$ is defined in \eqref{eq:Dab} and $\tr(a)$ is the
trace form, while if $\cJ$ denotes one of the Jordan algebras
$H_3(k)$, $H_3(\cQ)$, or $H_3(\cO)$, the subspaces $\Hom_{\Der
\cJ}(\cJ_0\otimes \cJ_0, \Der \cJ)$, $\Hom_{\Der\cJ}(\cJ_0\otimes
\cJ_0, k)$ and $\Hom_{\Der \cJ}(\cJ_0\otimes \cJ_0, \cJ_0)$ are
spanned by
$x\otimes y \mapsto d_{x,y}$, $x\otimes y \mapsto T(xy)$ and $x\otimes y \mapsto x\star y=x\bullet y-\frac 13 T(xy)1$, with
$d_{x,y}$ as in \eqref{eq:dxy} and $T(x)$ the generic trace. Then,
by imposing the Jacobi identity, it is easily checked that, up to
scalars, there exists only one way to introduce a Lie product in the
vector space $\bigl(\Der \cA \oplus \Der \J\bigr) \oplus
\bigl(\cA_0\otimes \J\bigr)$, for $\cA=\cQ$ or $\cA=\cO$, with the natural actions of the
derivation algebras on $\cA$ and $\J$. This product is given by
\begin{equation}\label{finalclasicaTits}
[a\otimes x, b\otimes y] =\frac {\alpha^2}3T(xy)D_{a,b} +
2\alpha^2\tr(ab)d_{x,y}+\alpha[a,b]\otimes x\star y
\end{equation}
where $\alpha \in k$. The resulting algebras for the same
ingredients and different nonzero scalars $\alpha$ are all
isomorphic and  hence isomorphic to the Classical Tits Construction
$\cT(\cO, \J)$  with $\J \neq H_3(\cK)$, or $\cT(\cQ,H_3(\cO))$.

For $\J=H_3(\cK)$ (which is isomorphic to the algebra $\Mat_3(k)$
with the symmetrized product), $\J_0$ is isomorphic to the adjoint
module $\Der \J$, and hence the subspaces $\Hom_{\Der\J}(\J_0\otimes
\J_0, \J_0)$ and $\Hom_{\Der\J}(\J_0\otimes \J_0, \Der \J)$ have
dimension $2$, being spanned by the symmetric product $x\star y$ and
the skew product $d_{x,y}$. Since the products in $\Hom_{\Der
\cO}(\cO_0\otimes \cO_0, \cO_0)$ are skew and symmetric in
$\Hom_{\Der \cO}(\cO_0\otimes \cO_0, k)$, the anticommutativity
imposed in the construction of a Lie algebra on the vector space
$\bigl(\Der \cO \oplus \Der \J\bigr) \oplus \bigl(\cO_0\otimes
\J\bigr)$ with the natural actions of the derivation algebras on
$\cO$ and $\J$,  can only be guaranteed if a symmetric product  in
$\Hom_{\Der\J}(\J_0\otimes \J_0, \J_0)$ and a skew-symmetric one in
$\Hom_{\Der\J}(\J_0\otimes \J_0, \Der \J)$ are used. This yields
again the Lie product in \eqref{finalclasicaTits} and, up to
isomorphism, the corresponding Classical Tits Construction
$\cT(\cO,H_3(\cK))$ given in Example \ref{ex:ClassicalTits}. This
provides cases (ii) and (iii) in the Theorem.
\end{proof}

\section*{Concluding remarks}

As mentioned in the Introduction, concerning the isotropy
irreducible homogeneous spaces, Wolf remarked in \cite{Wolf} that
only the irreducible homogeneous spaces $\mathbf{SO}(\dimens K)/\ad
K$ for an arbitrary compact simple Lie group follow a clear pattern.
These are related to the reductive pairs $(\frso(L),\ad L)$ for a
simple Lie algebra $L$, so $\ad L=\Der(L)=[\Der(L),\Der(L)]$, and
hence the reductive pair can be written as $(\frso(L),\Der(L))$.

The examples in Section \ref{Section:Examples} follow clear patterns
too. Moreover, a closer look at the classification of the non-simple
type irreducible LY-algebras shows that, apart from the irreducible
Lie triple systems and the exceptional cases that appear related to
the Classical Tits Construction in Theorem
\ref{th:irreduciblesexcepcionales}, there are two more classes, that
correspond to Examples \ref{ex:skewV1otimesV2} and
\ref{ex:slV1otimesslV2}.

Concerning the irreducible LY-algebras in Example
\ref{ex:skewV1otimesV2}, let $(V_1,\varphi_1)$ be a two dimensional
vector space endowed with a nonzero skew-symmetric bilinear form,
and let $(V_2,\varphi_2)$ be another vector space of dimension $\geq
3$ endowed with a nondegenerate $\epsilon$-symmetric bilinear form.
Then $T=V_1\otimes V_2$ is an irreducible Lie triple system, as in
Example \ref{ex:skewV1plusV2}, whose Lie algebra of derivations is
$\Der(T)=\frsp(V_1,\varphi_1)\oplus\anti(V_2,\varphi_2)$. Hence, the
reductive pair $(\g,\h)$ in Example \ref{ex:skewV1otimesV2} (or in
Theorem \ref{th:irreduciblesclasicas}, items (iv) and (v)), is
nothing else but
$\bigl(\anti(T,\varphi_1\otimes\varphi_2),\Der(T)\bigr)$.

Also, in Example \ref{ex:slV1otimesslV2} (or the first item in
Theorem \ref{th:irreduciblesclasicas}) two vector spaces $V_1$ and
$V_2$ of dimension $n_1$ and $n_2$ are considered. The tensor
product $V_1\otimes V_2$ can be identified to $k^{n_1}\otimes
k^{n_2}$ or to the space of rectangular matrices $V=\Mat_{n_1\times
n_2}(k)$. The pair $\cV=(V,V)$ is a Jordan pair (see \cite{Loos})
under the product given by $\{xyz\}=xy^tz+zy^tx$ for any $x,y,z\in
V$. The Lie algebra of derivations is
$\Der(\cV)=\frsl_{n_1}(k)\oplus\frsl_{n_2}(k)\oplus k$, which acts
naturally on $V$, and then its derived algebra is
$\Der_0(\cV)=[\Der(\cV),\Der(\cV)]=\frsl_{n_1}(k)\oplus\frsl_{n_2}(k)$.
Hence the reductive pair associated to the irreducible LY-algebra in
Example \ref{ex:slV1otimesslV2} is the pair
$\bigl(\frsl(V),\Der_0(\cV)\bigr)$.

This sort of patterns will explain most of the situations that arise
in the generic case \cite{forthcoming}.


\providecommand{\bysame}{\leavevmode\hbox
to3em{\hrulefill}\thinspace}
\providecommand{\MR}{\relax\ifhmode\unskip\space\fi MR }
\providecommand{\MRhref}[2]{%
  \href{http://www.ams.org/mathscinet-getitem?mr=#1}{#2}
} \providecommand{\href}[2]{#2}

\end{document}